\documentclass{article}

\usepackage{arxiv}

\usepackage[utf8]{inputenc} 
\usepackage[T1]{fontenc}    
\usepackage{hyperref}       
\usepackage{url}            
\usepackage{booktabs}       
\usepackage{amsfonts}       
\usepackage{nicefrac}       
\usepackage{microtype}      
\usepackage{lipsum}
\usepackage{amsmath}

\usepackage{amsfonts,amssymb,stmaryrd}
\usepackage{graphicx}
\usepackage{amsthm}
\usepackage{xcolor}
\usepackage{dirtytalk}

\newcommand{\R}{\mathbb{R}}

\newtheorem{theorem}{Theorem}[section]
\newtheorem{proposition}{Proposition}[section]
\newtheorem{lemma}{Lemma}[section]
\newtheorem{corollary}{Corollary}[section]
\newtheorem{conjecture}[subsection]{Conjecture}

\newtheorem{remark}{Remark}[section]

\usepackage{subcaption}
\usepackage{cite}
\usepackage{authblk}
\newcommand{\p}{\partial}
\newcommand{\bb}{\begin{equation}}
\newcommand{\ee}{\end{equation}}
\newcommand{\ba}{\begin{array}}
\newcommand{\ea}{\end{array}}
\newcommand{\f}{\frac}
\usepackage[all]{xy}
\newcommand{\ds}{\displaystyle}

\newcommand{\al}{\alpha}

\newcommand{\sign}{\text{sgn}\,}
\newcommand{\I}{{\cal I}_0}
\numberwithin{equation}{subsection}
\usepackage{csquotes}

\usepackage{ulem}

\title{Existence, continuation, persistence and dynamics of solutions for a generalized 0-Holm-Staley equation}

\author{
 Priscila~Leal~da Silva\thanks{priscila.silva@ufabc.edu.br or pri.leal.silva@gmail.com,\quad \href{https://orcid.org/0000-0001-6499-9480}{ORCID 0000-0001-6499-9480}}\,\,\, and\,\,\,  Igor~Leite~Freire\thanks{igor.freire@ufabc.edu.br or igor.leite.freire@gmail.com,\quad \href{http://orcid.org/0000-0003-0539-9674}{ORCID 0000-0003-0539-9674}} \\
  Centro de Matem\'atica, Computa\c{c}\~ao e Cogni\c{c}\~ao\\
 Universidade Federal do ABC\\
 Santo Andr\'e, Brazil}
 
\begin{document}
\maketitle

\begin{abstract}

We consider a family of non-local evolution equations including the $0-$Holm-Staley equation. We show that the family considered does not posses compactly supported solutions as long as the initial data is non-trivial. Also, we prove different unique continuation results for the solutions of the family studied. In addition, some special solutions, such as peakons and kinks, are studied and their dynamics are analyzed. Persistence properties of the solutions are also investigated as well as we describe the scenario for the global existence of solutions of the $0-$Holm-Staley equation. In particular, the prove of global existence of solutions as well as our demonstrations for unique continuation results of solutions partially answer some questions pointed out in [A. A. Himonas and R. C. Thompson, Persistence properties and unique continuation for a generalized Camassa-Holm equation, J. Math. Phys., vol. 55, paper 091503, (2014)].

\end{abstract}

{\bf MSC classification 2010:} 35A01, 74G25, 37K40, 35Q51.

\newpage
\tableofcontents
\newpage
\section{Introduction and motivation of the work}\label{sec1}

In \cite{anco} the equation (up to notation)
\bb\label{1.0.1}
u_t-u_{txx}+au^ku_x-b u^{k-1}u_xu_{xx}-cu^ku_{xxx}=0,
\ee
where $(ab,ac)\neq(0,0)$ and $k\neq0$, was considered from the point of view of conserved currents, point symmetries and peakon solutions. With these restrictions on the parameters, equation \eqref{1.0.1} is invariant under translations in $t$, $x$, scalings $(t,x,u)\mapsto (\lambda^{-k}t,x,\lambda u)$, $\lambda>0$, and if $k=1$ and $a=c$ we also have invariance under the Galilean boost $(t,x,u)\mapsto(t,x+\epsilon a t,u+\epsilon)$, see \cite[Proposition 1.1 ]{anco}.

In the same paper, conserved currents for \eqref{1.0.1} were also considered, see  \cite[Theorem 2.1]{anco}. Two of them are important to the present work, namely,
\bb\label{1.0.2}
\ba{lcl}
C^0&=&u,\\
\\
C^1&=&\ds{\f{a}{k+1}u^{k+1}+\f{kc-b}{2}u_x^2-cu^ku_{xx}-u_{tx}},
\ea
\ee
for $k=1$ or $b=kc$, and
\bb\label{1.0.3}
\ba{lcl}
C^0&=&\ds{\f{u^2+u_x^2}{2}},\\
\\
C^1&=&\ds{\left(\f{a}{k+2}u-cu_{xx}\right)u^{k+1}-uu_{tx}},
\ea
\ee
if and only if $b=(k+1)c$.

The relevance of the conserved currents is the following: if $(C^0,C^1)$ is a conserved current for \eqref{1.0.1}, then 
\bb\label{1.0.4}
\left.\left(\p_t C^0+\p_x C^1\right)\right|_{on\,\,\eqref{1.0.1}}\equiv0,
\ee
meaning that the divergence of the conserved currents vanishes identically on the solutions of the equation. This implies that the functional
$$
u\mapsto \mathcal{ H}[u]=\int_\R C^0 dx
$$
is a constant (of motion), or a {\it conserved quantity}, for the equation. Very often the last integral is also referred by analysts as {\it conservation law} for the equation. 
For further details, see \cite{anco,priigorjde} and references thereof.

Taking $a=(b+c)$ in \eqref{1.0.1} we obtain
\bb\label{1.0.5}
u_t-u_{txx}+(b+c)u^ku_x=bu^{k-1}u_{x}u_{xx}+cu^ku_{xxx},
\ee
which was considered in \cite{yan} (for $c=1$, see \cite{himade,himjmp2014}). We observe that \eqref{1.0.1} with $a=c=1$ and $b=0$ gives (note that if $b=0$ and $a=c\neq0$ we can always proceed with a scaling in $t$ and take $c=1$)
\bb\label{1.0.6}
m_t+u^km_x=0,\quad m=u(t,x)-u_{xx}(t,x),\quad t>0,
\ee
and for $k=1$ equation \eqref{1.0.6} is reduced to the equation $m_t+um_x=0$, which is a very particular case of the $b-$equation 
\bb\label{1.0.7}
m_t+um_x+b u_xm=0
\ee
introduced in \cite{deg}, later investigated in \cite{holm-siam,holm-pla} by Holm and Staley, and sometimes referred as Holm-Staley (HS) equation. For this reason we shall refer to \eqref{1.0.6} with arbitrary power as generalized $0$-Holm-Staley equation, or simply $g0-HS$ equation for short. 

It is worth mentioning that \eqref{1.0.6} with $k=1$ can also be obtained from shallow water elevation equations via Kodama transformation, see \cite{dullin-fluid,holm-physica}, which shows its relevance in the study of shallow water models. Solutions of \eqref{1.0.6} with $k=1$ (or \eqref{1.0.7} with $b=0$) were considered in \cite{holm-siam,holm-pla,qiao}.

More recently, wave-breaking and global existence of solutions for \eqref{1.0.5} were considered in \cite{yan}. However, some of the results proved there were done with the restriction $b\neq0$. Also, in \cite{himonas-jnl} ill-posedness for the $b-$equation \eqref{1.0.7} was also considered when $b>1$. 

We note that the results in \cite{anco,yan,himonas-jnl} {\it suggest} that the cases $b=0$ or $k=0$ make \eqref{1.0.1} very peculiar. This is also reinforced by the results of \cite{zhou}, where the solutions of \eqref{1.0.7} were studied and the case $b=0$ was excluded in some analysis, such as in theorems $2.1$ and $2.2$, concerned with non-existence of global solutions and their blow up, respectively.

It is also intriguing that \eqref{1.0.1} does not have conserved currents up to second order for $b=0$ and $k\notin\{-1,0,1\}$ (see \cite[Theorem 2.1]{anco}), a fact also observed in \cite{holm-siam} when \eqref{1.0.6} was considered with $k=1$. All of these results make us conjecture that no further conservation laws can be obtained to \eqref{1.0.6} beyond those reported in \cite{anco}.

For equations of the type \eqref{1.0.1}, the conserved quantities provide qualitative information about its solutions subject to an initial condition $u(0,x)=u_0(x)$. For example, if $b=0$ and $a=c=k=1$, then \eqref{1.0.1} has the conserved quantity
\bb\label{1.0.8}
\mathcal{ H}_0[u]=\int_\R udx.
\ee

It means that if $u$ does not change its sign, then the $L^1(\R)-$norm of the rapidly decaying solutions of \eqref{1.0.1} with $b=0$ and $a=c=k=1$ is conserved (this will be better explored in Theorem \ref{teo3.1} in Subsection \ref{subsec3.1}). On the other hand, if $k=-1$, then the equation \eqref{1.0.6} has the conserved quantity
\bb\label{1.0.9}
\mathcal{ H}[u]=\f{1}{2}\int_\R(u^2+u_x^2)dx,
\ee
which is essentially the square of the $H^1(\R)-$norm of the solution $u$ of the equation and, therefore, for solutions decaying to $0$ as $x\rightarrow\pm\infty$, their $H^1(\R)-$norms are conserved.

The aim of the present paper is to consider the Cauchy problem
\bb\label{1.0.10}
\left\{
\ba{l}
u_t-u_{txx}+u^ku_x-u^ku_{xxx}=0,\\
\\
u(0,x)=u_0(x),
\ea
\right.
\ee
and determine properties and behaviour of its solutions, as well as some particular solutions of \eqref{1.0.6}.

In \cite[Theorem 2.1]{yan} it was established the local well-posedness of the equation \eqref{1.0.5} with initial data $u(0,x)=u_0(x)$, where $u_0\in B_{p,r}^s(\R)$ ($B_{p,r}^s(\R)$ denotes a Besov space, see \cite{yan} for further details). Taking $p=r=2$, we can ensure local well-posedness to \eqref{1.0.5} with $u(0,\cdot)=u_0(\cdot)\in H^s(\R)$, $s>3/2$, as shown \cite[Corollary 2.1]{yan}, see also \cite[Theorem 1.1]{himjmp2014}. Therefore the local well-posedness for \eqref{1.0.10} is proved directly by invoking these results and, therefore, its demonstration is omitted. 

We note that in \cite{yan} the question of global existence of solutions to \eqref{1.0.5} with a certain choice of the parameters and the initial data is addressed, but not for equation \eqref{1.0.6}, meaning that while in \cite{yan} we have the local existence for \eqref{1.0.10}, its global existence is not considered, see \cite[Theorem 4.1]{yan}. Also, in the same reference the problem of blow up is considered. In fact, it was shown that the first blow up of \eqref{1.0.5} occurs only as a wave-breaking. Likewise in the case of global existence, the results for wave-breaking proved in \cite[Theorem 5.1]{yan} are not applicable to \eqref{1.0.10}.

We would like to observe that in \cite{himjmp2014} the authors considered unique continuation results and persistence properties for equation \eqref{1.0.5} with $c=1$ (called by them as {\it g-kbCH equation}. Some of their results, such as \cite[Theorem 1.2]{himjmp2014} deals with the situation $k=1$ and $b\in[0,3]$ or $b=k$ and $k$ is a positive odd integer. In this work we prove an analogous result for \eqref{1.0.5} with $c=0$ and $k$ is an arbitrary positive integer. 

The mentioned paper by Himonas and Thompson has some interesting open problems. For example,  immediately after \cite[Theorem 1.4]{himjmp2014} we have the following observation:

\begin{displayquote}Therefore, the question of whether the property of unique continuation is present for the Novikov equation or any other member of the g-kbCH family of equations not included in Theorem 1.2 is an interesting open problem. \end{displayquote}

Additionally, we have another open problem pointed out in \cite[page 3]{himjmp2014}:

\begin{displayquote}Similarly, the existence of global solutions for g-kbCH when $b\neq k+1$ or when $m_0$ changes sign, like in McKean$^{38,39}$ for CH, is another open problem.\end{displayquote}

We would like to mention that references $38$ and $39$ mentioned above corresponds to references \cite{mc1} and \cite{mc2} of the present work.

In our paper we shed light to these questions, extending the results proved in \cite{himjmp2014} to other cases, and go beyond: we also establish continuation results for the equation, as well as we show that as long as the initial data for \eqref{1.0.6} is non-zero, then the solutions of the equation, under certain conditions, cannot be compactly supported. In addition, we also improve results of \eqref{1.0.6} concerned with existence of global solutions.

In the next section we present our main results and show how they are inserted in the state of the art of the field.

\section{Notation, main results and outline of the paper}\label{sec2}

In this section we present the notation of the manuscript, as well as its main results, structure, novelties and challenges.

\subsection{Notation} 
Throughout this paper $\mathbb{Z}$ and $\mathbb{N}$ denote the sets of the integer and natural numbers, respectively, while $\mathbb{N}_0:=\{0\}\cup\mathbb{N}$. Given $s\in\R$, by $H^s(\mathbb{\R})$ we mean the usual Sobolev space of order $s$, with corresponding Sobolev norm denoted by $\|\cdot\|_{H^s}$, whereas $\|\cdot\|_p$, $1\leq p\leq\infty$, denotes the norm of the $L^p(\R)$ space. Given two functions $f$ and $g$, their convolution is denoted by $f\ast g$. If $u=u(t,x)$, we denote by $u_0(x)$ the function $x\mapsto u(0,x)$, $m=u-u_{xx}$ and $m_0=u-u_0''$. Note that $m=(1-\p_x^2)u$ and then $u=g\ast m$, where $g(x)=e^{-|x|}/2$. Of great importance for us is the fact that if $s\geq t$, then $H^s(\R)\hookrightarrow H^t(\R)$. 

We also recall that
$|f(x)|\sim O(g(x))$ as $x\nnearrow\infty$ (respectively $|x|\rightarrow\infty$) if there exists a real constant $L$ such that
$$
\lim_{x\rightarrow\infty}\f{|f(x)|}{g(x)}=L\quad\Big(\text{respectively } \lim_{|x|\rightarrow\infty}\f{|f(x)|}{g(x)}=L\Big),
$$
whereas $|f(x)|\sim o(g(x))$ as $x\nnearrow\infty$ (respectively $|x|\rightarrow\infty$) if
$$
\lim_{x\rightarrow\infty}\f{|f(x)|}{g(x)}=0\quad\Big(\text{respectively } \lim_{|x|\rightarrow\infty}\f{|f(x)|}{g(x)}=0\Big).
$$

Note that these two conditions are not mutually exclusive, but not equivalent: in case $L>0$, the former condition does not imply the latter, while the converse is true in case $L=0$.

Finally, in some parts of the paper we use the regularization $\sign{(0)}=0$, whereas $T$ denotes the lifespan of the solutions $u$ of the problem \eqref{1.0.10}. In particular, $T>0$.

\subsection{Main results} 

Our first result regarding \eqref{1.0.10} is:
\begin{theorem}\label{teo2.1}
Given an initial data $u_0\in H^{3}(\R)$, let $u$ be the corresponding solution of \eqref{1.0.10}, where $k$ is a positive integer.
\begin{enumerate}
    \item If $m_0$ does not change sign, then $m$ does not as well. Moreover, $\sign{(m)}=\sign{(m_0)}$.
    \item The momentum $m$ is compactly supported if and only if $m_0$ is compactly supported.
    \item If $m_0\geq0$ or $m_0\leq0$, then $u(t,x)\geq0$ or $u(t,x)\leq0$, respectively. 
    \item If $m_0\geq0$ or $m_0\leq0$, then $(u+u_x)(t,x)\geq0$ or $(u+u_x)(t,x)\leq0$, respectively.
    \end{enumerate}
\end{theorem}

In the literature of Camassa-Holm type equations, e.g, see \cite{const, himjmp2014,him-cmp}, very often if the sign of the initial momentum does not change, then usually this property persists in the corresponding solution. Theorem \ref{teo2.1} says that the same also holds for \eqref{1.0.10}. An example of initial momentum leading to a non-negative solution is the bump function
$$
m_0(x)=\left\{
\ba{lcl}
\ds{e^{-\f{1}{x^2-1}}},&\quad& |x|<1,\\
\\
0,&\quad & |x|\geq1.
\ea\right.
$$

Note, however, that the converse is not true: while the positiveness of the initial momentum implies that the corresponding solution will have the same property (including the initial data), a positive initial data would not necessarily imply that the corresponding initial momentum is positive. In fact, it is enough to take $u_0$ as the bump function above.

For $k=1$ we are able to prove a continuation result for solutions of \eqref{1.0.10}, as stated in the next result.

\begin{theorem}\label{teo2.2}
Let $T>0$, $I\subseteq\R$ a non-empty open interval, $\Omega:=(0,T)\times I$, and assume that $u\in C^0([0,T), H^s(\R))$, $s>3/2$, is a solution of the equation 
\bb\label{2.2.1}
u_t-u_{txx}+uu_x-uu_{xxx}=0.
\ee
If $u\big|_{\Omega}=0$, then $u\equiv0$ on $[0,T)\times\R$ and, moreover, $u$ can be extended globally.
\end{theorem}

It is possible to relax the condition that $u$ vanishes on $\Omega=(0,T)\times I$ in Theorem \ref{teo2.2}, for some open set $I\subseteq\R$. In fact, we can prove a similar result on an arbitrary, non-empty open set $\Omega\subseteq(0,T)\times \R$, but the price, however, is to impose that the solution does not change its sign.

\begin{theorem}\label{teo2.3}
Let $t_0,t_1,T\in\R$ such that $0<t_0<t_1<T$, $I\subseteq\R$ a non-empty open interval, and $\Omega:=(t_0,t_1)\times I$. Suppose that $u\in C^0([0,T), H^s(\R))$, $s>3/2$, is a solution of \eqref{2.2.1}. If $u$ is either non-negative or non-positive and $u\big|_\Omega=0$, then $u\equiv0$ on $[0,T)\times\R$ and, moreover, $u$ can be extended globally.
\end{theorem}

The following result is a foregone conclusion of the last theorem.

\begin{corollary}\label{cor2.1}
Assume that $u_0\in H^3(\R)$ and let $u\in C^0([0,T),H^3(\R))\cap C^1([0,T),H^2(\R))$ be the corresponding solution of \eqref{2.2.1}. Suppose that $m_0\in L^1(\R)$ and its sign does not change. If there exists a rectangle $\Omega=(t_0,t_1)\times (x_0,x_1)\subseteq[0,T)\times\R$ such that $u\big|_\Omega\equiv0$, then $u$ vanishes everywhere.
\end{corollary}

We have a very strong consequence of Theorem \ref{teo2.3}.

\begin{corollary}\label{cor2.2}
Assume that $u_0\in H^3(\R)$ is a non-vanishing compactly supported data for
\bb\label{2.2.2}
\left\{
\ba{l}
u_t-u_{txx}+uu_x-uu_{xxx}=0,\\
\\
u(0,x)=u_0(x),
\ea
\right.
\ee
such that $m_0$ does not change sign. Then the corresponding solution $u$ is not compactly supported.
\end{corollary}

Theorems \ref{teo2.1}--\ref{teo2.3} only request that the solution exists locally. We also observe that the local well-posedness assured by the results in \cite{yan,himjmp2014} guarantees the existence of a solution $u\in C^0([0,T);H^s(\R))\cap C^1([0,T);H^{s-1}(\R))$, for a certain $T>0$ and $s>3/2$. Two questions of capital importance are: Does this solution exist for $T=\infty$? Does this solution develop any singularity for $T<\infty$? The first question deals with the problem of {\it global existence}, whereas the second is related to the question of {\it blow up} in finite time, meaning that the solution becomes unbounded for finite values of $T$.

We can improve Yan's achievements \cite{yan} regarding the global existence of solutions of \eqref{2.2.1} with the following global existence result, in which \eqref{1.0.8} is of vital importance:
\begin{theorem}\label{teo2.4}
Given $u_0\in H^3(\R)$, let $u(t,\cdot)\in H^3(\R)$ be the corresponding unique solution of \eqref{2.2.2}. If $m_0 \in L^1(\R)\cap H^1(\R)$ does not change sign, then the solution $u$ exists globally in $C^0([0,\infty);H^3(\R))\cap C^1([0,\infty); H^2(\R)).$
\end{theorem}

The demonstration of Theorem \ref{teo2.4} provides a necessary condition for the wave-breaking of the solutions of \eqref{2.2.2}, as stated by the next result:

\begin{corollary}\label{cor2.3}
Let $u_0\in H^s(\R)$, $s>3/2$, and $T$ the lifespan of the corresponding solution $u(t,x)$ of \eqref{2.2.2}. If $\|u(t,x)\|_\infty$ is bounded as $t\nnearrow T$, then the slope $u_x(t,x)$ of the solution $u$ is bounded near $T$. In particular, there is no wave-breaking of the solutions.
\end{corollary}

We recall that for dealing with both global existence and wave-breaking problems, we {\it usually} need:
\begin{itemize}
    \item Local existence results established;
    \item Qualitative properties of the solutions, quite often manifested through conserved quantities or, which is the same, conserved currents for the equation must be known. In the absence of suitable conserved currents, other similar information, such as estimates on the solutions should be at our disposal.
\end{itemize}

Apart from the cases $k=\pm1$, \eqref{1.0.6} does not have other known conservation laws, which means that we do not have enough information to determine whether the local results can be extended to a global property for general $k$ using the conserved quantities. 

Even in the case $k=1$, the only known conservation law for the equation in \eqref{2.2.2} is useless for extending the local solution to global one. Therefore, in order to prove Theorem \ref{teo2.4} we show that the solution of \eqref{2.2.2} is bounded from above by the $H^3(\R)-$norm of the initial data, provided that the $x-$derivative of the solution is bounded from below. The last property comes from the fact that the sign of the initial momentum $m_0$ is invariant.

In \cite{qiao} it was shown that \eqref{2.2.1} has some particular travelling wave solutions called peakons shaping as the ones of the CH equation, that is, $u(t,x)=ce^{-|x-ct|}$. We will show in Section \ref{sec6} that \eqref{1.0.6} has the solution $u(t,x)=c^{1/k}e^{-|x-ct|}$. On the other hand, it is well know that under certain circumstances, if $u(t,x)$ is a solution of the CH equation, then for each fixed $t$, $x\mapsto u(t,x)$ behaves like a peakon solution, see \cite{him-cmp}. Our next result is in line with this fact, more precisely, it is concerned with persistence properties of the asymptotic behaviour of the solutions on compact subsets of $[0,T)$, where $T$ is the lifespan of the solution. Let $T_0\in(0,T)$ be an arbitrary value for which the solution exists and let $\I=[0,T_0]$. 

\begin{theorem}\label{teo2.5}
Assume that $u\in C^0(\I,H^s(\R))$, $s>3/2$, is a solution of \eqref{1.0.10}. If, for some $\theta\in(0,1)$, 
    \bb\label{2.2.3}
    |u_0(x)|\sim O(e^{-\theta |x|})\quad\text{and}\quad |u_0'(x)|\sim O(e^{-\theta |x|})\quad\text{as}\quad |x|\rightarrow\infty,
    \ee
then 
    \bb\label{2.2.4}
    |u(t,x)|\sim O(e^{-\theta |x|})\quad\text{and}\quad |u_x(t,x)|\sim O(e^{-\theta |x|})\quad\text{as}\quad |x|\rightarrow\infty,
    \ee
uniformly in $\I$.
\end{theorem}

The demonstration of Theorem \ref{teo2.5} is based on the works by Himonas and co-authors \cite{himjmp2014,him-cmp}. Our last theorem is a different unique continuation result for the solutions of \eqref{1.0.6}, whose demonstration is strongly dependent on Theorem \ref{teo2.5}.

\begin{theorem}\label{teo2.6}
Let $u\in C^0(\I,H^s(\R))$, $s\geq 3$, be a solution of \eqref{1.0.10}. Assume that:
\begin{enumerate}
    \item\label{teo1.6-2} For some $\al\in(1/(k+1),1)$, 
    \bb\label{2.2.5}
    |u_0(x)|\sim o(e^{-x})\quad\text{and}\quad |u_0'(x)|\sim O(e^{-\al x})\quad\text{as}\quad x\nnearrow\infty,
    \ee
    and
    \item\label{teo2.3-3} There exists $t_1\in\I$, $t_1>0$, such that
    \bb\label{2.2.6}
    |u(t_1,x)|\sim o(e^{-x})\quad\text{as}\quad x\nnearrow\infty.
    \ee

If
\begin{enumerate}
    \item $k=1$, then $u\equiv0$.
    \item $k$ is even and $m_0(x)\geq0$, for all $x\in\R$, then $u\equiv0$.
    \item $k$ is odd and either $m_0(x)\geq0$ or $m_0(x)\leq0$, for all $x\in\R$, then $u\equiv0$.
\end{enumerate}    
\end{enumerate}
\end{theorem}

As previously mentioned, in \cite[Theorem 1.2]{himjmp2014} they proved a similar result to our Theorem \ref{teo2.6}, but there are significant differences: \cite[Theorem 1.2]{himjmp2014} is concerned with \eqref{1.0.5} with $c=1$. Moreover, their theorem considered the situation $k=1$ and $b\in[0,3]$ or $b=k$ and $k$ is a positive odd integer. In our case we have proved a result for \eqref{1.0.5} with $c=0$ and $k$ is an arbitrary positive integer. The price we pay, however, is the imposition of some restrictions on the initial data.

\subsection{Organization of the paper} In Section \ref{sec3} we prove several technical results that will be useful in the demonstration of our main contributions. Next, in Section \ref{sec4}, we study the behavior of compactly supported data and provide the continuation of solutions for the case $k=1$ to prove theorems \ref{teo2.1}--\ref{teo2.3}. In Section \ref{sec5} we prove Theorem \ref{teo2.4}. In Section \ref{sec6} we study some special solutions of the equation \eqref{1.0.6}, more precisely, (multi-)peakons and other wave solutions of \eqref{1.0.6} for any integer $k$. For the cases $k=1$ and $k=-1$ we also use the conserved quantities \eqref{1.0.8} and \eqref{1.0.9}, respectively, to construct solutions compatible with them. Finally, in Section \ref{sec7} we prove theorems \ref{teo2.5} and \ref{teo2.6} and our discussions and conclusions are presented in sections \ref{sec8} and \ref{sec9}, respectively.

\subsection{Challenges and novelties of the paper} 

The first unique continuation results for equation \eqref{2.2.1}, given in Theorem \ref{teo2.2}, are based on some ideas introduced in \cite{linares} and the use of the conserved quantity \eqref{1.0.8}, as observed in \cite{igor-dgh}. The fact that the integrand in \eqref{1.0.8} is not necessarily positive nor negative brings some complications in the use of \eqref{1.0.8}. In order to overcome this problem we then find conditions for the solutions of \eqref{2.2.2} to not change their sign, which then implies that the integral kernel in \eqref{1.0.8} is either non-positive or non-negative. As a consequence of this fact we show that the $L^1(\R)-$norm of the solution and the corresponding momentum are conserved, see Theorem \ref{teo3.1}, which will be of great relevance to prove Theorem \ref{teo2.4}, that guarantees the global existence of solutions of the problem \eqref{2.2.2}. We observe that the Cauchy problem \eqref{2.2.2} has very little structure and the only structural property known for the equation in \eqref{2.2.2} is the invariant \eqref{1.0.8}, which makes the proof of global existence quite challenging. In order to prove it, we show that if the initial data is in $H^3(\R)$ and its momentum does not change sign, then the $x$ derivative of the solution $u$ of \eqref{2.2.2} is bounded from below by the negative of the $L^1(\R)-$norm of the initial momentum. This is enough to assure that the $H^3(\R)-$norm of the solution is bounded, for each $t\in\R$.

Very often, in our results we require that the solution $u$ does not change its sign. The question is: how can we guarantee that? We show in Theorem \ref{teo3.1} (in Section \ref{sec3}) that if the initial momentum $m_0$ does not change its sign, then such property is inherited by the corresponding solution.

Beyond the qualitative properties given in theorems \ref{teo2.1}--\ref{teo2.4}, we also consider peakon and cliff solutions of the equation \eqref{1.0.6}. We show that such solutions may exist for any integer $k\neq0$ (the case $k=0$ is not considered because the resulting equation is linear). We pay considerable attention to equation \eqref{1.0.10} with $k=-1$, which brings a considerable singularity to the problem, but has the $H^1(\R)-$norm of the solutions as a conserved quantity. We find explicit solutions showing peakon-peakon and peakon-antipeakon dynamics. As far as we know, this is the first time that a singular non-evolution equation of the type \eqref{1.0.6} has peakon solutions of the same shape of the one admitted by the Camassa-Holm equation \cite{chprl} reported, although some singular evolution equations having peakon type solutions are known, {\it e.g.}, see \cite{pri-proc}.

\section{Preliminaries and technical results}\label{sec3}

In this section we prove some technical results that will be relevant in the proofs of theorems \ref{teo2.1}--\ref{teo2.4}. We begin with the following:

\begin{lemma}\label{lema3.1}
Let $u=u(t,x)$ be a solution of 
\begin{align}\label{3.0.1}
    u_t-u_{txx} + uu_x-uu_{xxx}=0
\end{align}
such that $u(0,x)=:u_0(x)$ and $u(t,\cdot)$, $u_x(t,\cdot)$, $u_{xx}(t,\cdot)$ and $u_{tx}(t,\cdot)$ are integrable and vanish at $x=\pm\infty$ for all values of $t$ such that the solution exists. Then
\bb\label{3.0.2}
\int_\R udx=\int_\R u_0dx=\int_\R m dx=\int_\R m_0dx,
\ee
where $m_0:=u_0-u_0''$.
\end{lemma}
\begin{proof}
We note that \eqref{1.0.2} is a conserved vector for \eqref{3.0.1}, which means that
\bb\label{3.0.3}
\p_t(u)+\p_x\Big(\f{u^2}{2}+\f{u_x^2}{2}-uu_{xx}-u_{tx}\Big)=0
\ee
on the solutions of \eqref{2.2.1}.

Let $\mathcal{ H}_0[u]$ be given by \eqref{1.0.8}, where $u$ is the solution of \eqref{3.0.1} with the initial datum $u_0$. From \eqref{3.0.3} we have
$$
\f{d}{dt}\mathcal{ H}_0[u]=\f{d}{dt}\int_\R u\,dx=\int_\R u_t\,dx=-\left.\Big(\f{u^2}{2}+\f{u_x^2}{2}-uu_x-u_{tx}\Big)\right|_{-\infty}^\infty=0,
$$
which implies that $\mathcal{H}_0[u]:=\mathcal{H}_0[u(t,x)]=\mathcal{H}_0[u(0,x)]=:\mathcal{H}_0[u_0]$. This proves the first equality. 

Now we observe that
$$
\int_\R mdx=\int_\R (u-u_{xx})dx=\mathcal{ H}[u]-\int_\R u_{xx}dx=\mathcal{ H}[u]-\left.u_{x}\right|_{-\infty}^\infty=\mathcal{ H}[u],
$$
which is enough to prove \eqref{3.0.2}.
\end{proof}

Our next lemma is similar to \cite[Theorem 3.1]{const} and our demonstration follows closely the original ideas. The result in \cite[Corollary 2.1]{yan} assures that if $u_0\in H^3(\R)$, then we have a unique solution $u\in C^0([0,T),H^3(\R))\cap C^1([0,T),H^2(\R))$.

\begin{lemma}\label{lema3.2}
Given $u_0\in H^3(\R)$, let $u\in C^1([0,T),H^2(\R))$ be the corresponding unique solution of \eqref{1.0.1}. Then the initial value problem
\bb\label{3.0.4}
\begin{cases}
    \partial_t y(t,x) = u(t,y)^k,\\
    \\
    y(0,x) = x,
\end{cases}
\ee
where $k$ is a positive integer, has a unique solution $y(t,x)$ such that $y_x(t,x)>0$ for any $(t,x)\in [0,T)\times \R$. Moreover, for each $t\geq 0$ fixed, $y(t,\cdot)$ is an increasing diffeomorphism on the line.
\end{lemma}

\begin{proof}
Since $u\in C^1([0,T)\times\R)$, then both $u(t,\cdot)$ and $u_x(t,\cdot)$ are bounded and Lipschitz, while $u(\cdot,x)$ and $u_x(\cdot,x)$ are $C^1$. For each fixed $x\in\R$, the Picard-Lindelöf Theorem \cite[page 10]{valls} assures the existence of a unique continuous solution $y(\cdot,x)$ satisfying the problem \eqref{3.0.4} and defined on $[0,T)$, for some $T>0$.

If we let $x$ change, we can then differentiate \eqref{3.0.4} and obtain
$$
\begin{cases}
    \partial_t y_x(t,x) = k u(t,y(t,x))^{k-1}u_x(t,y(t,x))y_x(t,x),\\
    \\
    y_x(0,x) = 1.
\end{cases}
$$

Fixing $x$ and defining $T_x:=\sup\{t\in[0,T),\,\,y_x(t,x)>0\}$, for each $t\in[0,T_x)$, we have
\bb\label{3.0.5}
y_x(t,x)=\exp{\left(k u(t,y(t,x))^{k-1}u_x(t,y(t,x))\right)}>0.
\ee

The conditions on $u$ imply that $y_x(t,\cdot)$ is continuous. We claim that $T_x=T$. Actually, if it were not true, from the continuity of $y_x(t,\cdot)$ we would have $y_x(t,\overline{x})=0$ for some $\overline{x}\in\R$, which is a clear contradiction with \eqref{3.0.5}.

The continuity of $y(t,\cdot)$ implies that $J_t:=\{y(t,x),\,\,x\in\R\}$ is an interval. Again, by \eqref{3.0.5} we are forced to conclude that $y(t,\cdot)$ is a diffeomorphism between $\R$ and $J_t$. To conclude the demonstration we need to show that $J_t=\R$.

By the Sobolev Embedding Theorem \cite[p. 317]{taylor}, given $t\in(0,T)$, the function $u_x(s,z)$ is uniformly bounded for each $[s,z]\in[0,t]\times\R$ and, consequently, there exists a positive number $k_t>0$ such that $e^{-k_t}\leq y_x(t,x)\leq e^{k_t}$,
which, after integration, yields the inequality
$ e^{-k_t}x\leq y(t,x)\leq e^{k_t}x$. This implies that $J_t$ cannot have either lower or upper bounds.
\end{proof}

\begin{theorem}\label{teo3.1}Let $u_0\in H^3(\R)$ be an initial data for $(\ref{1.0.6})$, with corresponding solution $u$.
\begin{enumerate}
    \item\label{teo3.1-1} Assume that the sign of $m_0$ does not change. Then $\sign{(u)}=\sign{(u_0)}=\sign{(m)}=\sign{(m_0)}$ and they do not change;
    \item\label{teo3.1-2} Assume that $k=1$ and $m_0\in L^1(\R)$. Then $- u_{x}(t,x)\leq\|m_0\|_{1}$, for any $(t,x)\in[0,T)\times\R$.
\end{enumerate}
\end{theorem}
\begin{proof} Let $y$ be the diffeomorphism given in Lemma \ref{lema3.2}. Differentiation of $m(t,y(t,x))$ with respect to $t$ yields
$$\frac{d}{dt}m(t,y(t,x)) = m_t + y_tm_x = m_t + u^km_x = 0,$$
which means that $m(t,y(t,x))$ does not depend on $t$ and, therefore, 
\bb\label{3.0.6}
m(t,y(t,x)) = m_0(x).
\ee

Since $y$ is a diffeomorphism, we conclude that $\sign(m_0(\cdot)) = \sign(m(t,y(t,\cdot)))$. Therefore,  $m_0$ does not change sign if and only if $m$ does not change sign too. Now we observe that $u(t,x)=g\ast m(t,x)$, where $g(x)=e^{-|x|}/2$. Since $g(x)>0$, then $\sign{(u(t,x))}=\sign{(m(t,x))}$ and $\sign{(u_0(x))}=\sign{(m_0(x))}$. This proves $\ref{teo3.1-1}$. To prove the second part, let us first assume $m_0\geq0$. From \eqref{3.0.2}, we have
$$
\|m_0\|_1=\int_\R m_0(s)\,ds=\int_{\R}m(t,s)\,ds\geq\int_{-\infty}^x m(t,s)\,ds=\left(\int_{-\infty}^x u(t,s)ds\right)- u_x(x,t). 
$$
Since $u\geq0$ and
$$
0\leq\int_{-\infty}^xu(t,s)\,ds\leq\int_\R u(t,s)\,ds<\infty,
$$
we have $\|m_0\|_1\geq - u_{x}(t,x)$. 

Let us now prove the inequality whenever $m_0\leq0$. In this case, we have $-m\geq 0$ and $-u\geq 0$ as well. Then
$$
0\leq-\int_{\R} u(t,s) ds = -\int_{\R} m(t,s) ds = -\int_{\R}m_0(s)ds < \infty.
$$
Since
$$0\geq \int_{-\infty}^xm(t,s)ds = \int_{-\infty}^x(u- u_{xx})(t,s)ds = \int_{-\infty}^x u(t,s)ds-u_{x}(t,x),$$
we have
$$
- u_x(x,t)\leq -\int_{-\infty}^x u(t,s)\,ds\leq \int_\R -u(t,s)\,ds=\int_\R -m_0(s)\,ds=\|m_0\|_1, 
$$
which proves the result.
\end{proof}

\section{Continuation and compactly supported data}\label{sec4}

Here we prove theorems \ref{teo2.1}--\ref{teo2.3}. 

\subsection{Proof of the Theorem \ref{teo2.1}}\label{subsec3.1}

Theorem \ref{teo2.1} is a an immediate recollection of results proven so far. In fact, if $m_0$ does not change sign, then Theorem \ref{teo3.1} concludes the proof of item 1. Moreover, if we assume that $m_0$ is compactly supported on $[a,b]$, from \eqref{3.0.6} we conclude that $m$ is supported on $[y(t,a),y(t,b)]$. Conversely, fix $t>0$. If $m(t,x)$ is compactly supported on $[a,b]$, then $m_0$ vanishes identically outside the interval $[y^{-1}(t,a),y^{-1}(t,b)]$, which proves part 2 of Theorem \ref{teo2.1}.


Part 3 follows from \eqref{3.0.6} and the relation $u(t,x)=\Lambda^{-2}m(t,x)$. Finally, writing explicitly this convolution, we have
\bb\label{4.1.1}
\ba{lcl}
u(t,x)&=&\ds{\f{e^{-x}}{2}\int_{-\infty}^xe^sm(t,s)ds+\f{e^{x}}{2}\int_x^{\infty}e^{-s}m(t,s)ds,}\\
\\
\text{and}\\
\\
u_x(t,x)&=&\ds{-\f{e^{-x}}{2}\int_{-\infty}^xe^sm(t,s)ds+\f{e^{x}}{2}\int^{\infty}_xe^{-s}m(t,s)ds.}
\ea
\ee

Adding them, we have
\bb\label{4.1.2}
(u+u_x)(t,x)=e^x\int^{\infty}_xe^{-s}m(t,s)ds,
\ee
which proves the remaining part of the theorem. \hfill$\square$

In line with Theorem \ref{teo2.1} we have the following result.

\begin{theorem}\label{teo4.1}
Assume that $u_0\in H^3(\R)$ is such that $m_0$ is either non-negative or non-positive and let $u$ be the corresponding solution of \eqref{2.2.2}. Then the quantities $\|u(t,\cdot)\|_1$ and $\|m(t,\cdot)\|_1$ are constant.
\end{theorem}
\begin{proof}
Let us prove that $\|u(t,\cdot)\|_1$ and $\|m(t,\cdot)\|_1$ are conserved for any $t\in[0,T)$. Firstly, assume that $m_0\geq0$. Then $u\geq0$ and the conserved quantity \eqref{1.0.8} yields
$$
{\cal H}_0[u]=\int_\R u(t,x)dx=\int_\R|u(t,x)|dx=\|u(t,\cdot)\|_1,
$$
which means that $\|u(t,\cdot)\|_1=\|u_0\|_1$. On the other hand, if $u\leq0$, then
$$
{\cal H}_0[u]=-\int_\R \left(-u(t,x)\right)dx=-\int_\R|u(t,x)|dx=-\|u(t,\cdot)\|_1,
$$
and, again, $\|u(t,\cdot)\|_1=\|u_0\|_1$.

By Lemma \ref{lema3.1} we know that
$$
\int_\R m(t,x)dx=\int_\R u(t,x)dx,
$$
and the remaining part of the demonstration is analogous, the reason why we omit it.
\end{proof}

\subsection{Proof of Theorem \ref{teo2.2}}

We begin by recalling that if $u_0\in H^s(\R)$, with $s>3/2$, then the Cauchy problem of \eqref{1.0.7} with $u$ satisfying $u(0,x)=u_0(x)$ has a unique local solution $u\in C^0([0,T),H^s(\R))\cap C^1([0,T),H^{s-1}(\R))$, see \cite[Corollary 2.1]{yan} and \cite[Theorem 1.1]{himjmp2014}. Instead of proving Theorem \ref{teo2.2} directly, we shall prove the following stronger result regarding the $b-$equation \eqref{1.0.7}.
\begin{theorem}\label{teo3.2}
Let $b\in[0,3]$, $s>3/2$, $u_0\in H^s(\R)$, $u$ be the corresponding solution of \eqref{1.0.7} satisfying $u(0,x)=u_0(x)$, $I\subseteq\R$ be a non-empty open interval and $\Omega=(0,T)\times I$, where $T$ is the lifespan of the solution $u$. If $\left.u\right|_{\Omega}\equiv0$, then $u\equiv0$ on $[0,T)\times\R$. Moreover, the solution can be extended globally.
\end{theorem}
\begin{proof}
The fact that $u$ can be extended globally is immediate once we prove the result, since we take $u(t,x)=0$, for all $(t,x)$.

Note that \eqref{1.0.7} can be rewritten as
\bb\label{4.2.1}
u_t+uu_x+\p_x\Lambda^{-2}\left(\f{b}{2}u^2+\f{3-b}{2}u_x^2\right)=0.
\ee

Fix $t_0\in(0,T)$ and define $F:\R\rightarrow\R$ by
\bb\label{4.2.2}
F(x):=\p_x\Lambda^{-2}\left(\f{b}{2}u(t_0,x)^2+\f{3-b}{2}u_x(t_0,x)^2\right).
\ee

We observe that
\begin{itemize}
    \item The conditions on $u$ and the definition of $F$ imply that $F\in C^1(\R)$;
    \item If $x\in I$, by \eqref{4.2.1} and \eqref{4.2.2} we have
    $$
    F(x)=\p_x\Lambda^{-2}\left(\f{b}{2}u(t_0,x)^2+\f{3-b}{2}u_x(t_0,x)^2\right)=-\left(u_t+uu_x\right)(t_0,x)\equiv0.
    $$
    \item By the Fundamental Theorem of Calculus, given $x_0,\,x_1\in I$, with $x_0<x_1$, we have
    $$
    0=F(x_1)-F(x_0)=\int_{x_0}^{x_1} F'(x)dx.
    $$
    \item We have the identity $\p_x^2\Lambda^{-2}=\p_x^2\Lambda^{-2}-1$.
    \item We note that if $b\in[0,3]$ and 
    $$f(x):=\f{b}{2}u(t_0,x)^2+\f{3-b}{2}u_x(t_0,x)^2,$$
    then $f$ is non-negative, continuous, and 
    $$
    (\Lambda^{-2}f)(x)=(g\ast f)(x)=\int_{-\infty}^{+\infty}\f{e^{-|x-y|}}{2}f(y)dy\geq 0.
    $$
    Moreover, the last integral vanishes if and only if $f(x)\equiv0$, for all $x\in\R$.

\end{itemize}

From the observations above, we have
$$
0=F(x_1)-F(x_0)=\int_{x_0}^{x_1}F'(x)dx=\int_{x_0}^{x_1}\left(\Lambda^{-2}(f)(x)-\underbrace{f(x)}_{=0,\,\,\text{on}\,I}\right)dx=\int_{x_0}^{x_1}(g\ast f)(x)dx,
$$
which implies that 
\bb\label{4.2.3}
\f{b}{2}u(t_0,x)^2+\f{3-b}{2}u_x(t_0,x)^2=0.
\ee
If $b\in(0,3]$, then \eqref{4.2.3} implies that $u(t_0,x)=0$. If $b=0$, we are forced to conclude that $u(t_0,x)=c$, for some constant $c$. Since $u\rightarrow0$ as $|x|\rightarrow\infty$, we conclude again that $u(t_0,x)=0$.

This proves that for each $t_0\in(0,T)$, then $x\mapsto u(t_0,x)$ vanishes. Therefore, the solution vanishes on $(0,T)\times\R$ and, by continuity, on $[0,T)\times\R$.
\end{proof}

{\bf Proof of Theorem \ref{teo2.2}.} The proof of Theorem \ref{teo2.2} is nothing but an immediate corollary of Theorem \ref{teo3.2} with $b=0$.\hfill$\square$

\subsection{Proof of the Theorem \ref{teo2.3}}

We now present the demonstration of Theorem \eqref{teo2.3} and its corollaries.

{\bf Proof of Theorem \ref{teo2.3}.} To prove Theorem \ref{teo2.3}, we claim that if $u$ vanishes on $(t_0,t_1)\times I$, then it vanishes on $(t_0,t_1)\times\R$. In order to prove it, consider the function \eqref{4.2.2} with $b=0$. Proceeding similarly as in the demonstration of Theorem \ref{teo3.2}, for each $t^\ast\in(t_0,t_1)$ fixed, we conclude that $u_x(t^\ast,x)=0$, for all $x\in\R$, which forces $u(t^\ast,x)=0$, for all $x\in\R$. Since this holds to all $t^\ast\in(t_0,t_1)$, we conclude our claim.

For any $t^{\ast}\in (t_0,t_1)$, we have $u(t^{\ast},x)=0,$ which means that $\mathcal{H}_0[u(t^{\ast},\cdot)] = 0$, where $\mathcal{H}_0[u]$ is given by \eqref{1.0.8}. In view of the conservation of $\mathcal{H}_0[u]$, this implies that $\mathcal{H}_0[u(t,\cdot)]=0$ for every $t$ such that the solution is defined. On the other hand, since the sign of $u$ does not change, we note that by Theorem \ref{teo4.1} the conserved quantity $\eqref{1.0.8}$ vanishes if and only if $u(t,\cdot)\equiv0$, which says that $u\equiv0$ on $[0,T)\times\R$ (and, therefore, it can be extended globally), finishing the proof of Theorem \ref{teo2.3}. \hfill$\square$

{\bf Proof Corollary \ref{cor2.1}.} The conditions on $m_0$ implies that the corresponding solution $u$ of the associated Cauchy problem does not change its sign. The result is an immediate consequence of Theorem \ref{teo2.3}. \hfill$\square$

{\bf Proof Corollary \ref{cor2.2}.} Assume that $0\not\equiv u_0\in H^3(\R)$ is a compactly supported initial data for the problem \eqref{2.2.2}, and $u$ its corresponding solution. If $u$ were compactly supported, then for each $t_0\in(0,T)$ fixed, we would be able to find numbers $a_{t_0}$ and $b_{t_0}$, with $a_{t_0}<b_{t_0}$, such that $u(t_0,x)=0$, $x\in[a_{t_0},b_{t_0}]$. Letting $x\in[a_{t_0},b_{t_0}]$, from the equation in \eqref{2.2.2} we obtain
$$
0=(u_t-u_{txx})(t_0,x)+u(t_0,x)(u_x-u_{xxx})(t_0,x)=(1-\p_x^2)u_t(t_0,x)
$$
and then, $u_t(t_0,x)=0$. Therefore, \eqref{4.2.2} and \eqref{4.2.1} with $b=1$ give $F(x)=0$, $x\in[a_{t_0},b_{t_0}]$. Proceeding similarly as in the demonstration of Theorem \ref{teo3.2} we conclude that $u(t_0,x)=0$, $x\in\R$. Since $m_0$ does not change sign, Theorem \ref{teo4.1} implies that $\|u(t,\cdot)\|_1$ is the same for any value of $t$. Therefore, if $u_0\neq0$, we have $0=\|u(t_0,\cdot)\|_1=\|u_0\|_1>0$, which is a contradiction. \hfill$\square$

\section{Global existence of solutions}\label{sec5}

We begin with the following result.
\begin{lemma}\label{lema5.1}
Given $u_0\in H^3(\R)$, let $u\in C^0([0,T),H^3(\R))$ be the unique solution of \eqref{3.0.1} subject to $u(0,x)=u_0(x)$, for some $T>0$. If there exists a positive constant $\kappa$ such that $u_x>-\kappa$, then $\Vert u\Vert_{H^3}\leq e^{\kappa t/2}\Vert u_0\Vert_{H^3}$.
\end{lemma}

\begin{proof}
We begin with the observation that if $u\equiv0$, then the result is trivial. Therefore, we assume that $u_0\not\equiv0$.

Let $u_0$ and $u$ be given as enunciated. Then we have the following identities:
\begin{itemize}
\item After multiplying \eqref{3.0.1} by $u$ and some manipulation, we obtain
\begin{align}\label{5.0.1}
\partial_t\left(\frac{u^2+u_x^2}{2}\right)+\partial_x\left(\frac{u^3}{3} - uu_{tx} - u^2u_{xx}\right) + u\p_x (u_x^2)=0.
\end{align}

\item Now calculating the $x$ derivative of \eqref{3.0.1} and multiplying the result by $u_x$ we have
\bb\label{5.0.2}
\ba{l}
    \ds{\partial_t\left(\frac{u_x^2+u_{xx}^2}{2}\right) - \partial_x(u_xu_{txx} + uu_xu_{xxx})}\\
    \\
    \ds{+u\frac{1}{2}\partial_x(u_x^2+ u_{xx}^2)+ u_x^3=0.}
\ea
\ee

\item Finally, differentiating \eqref{3.0.1} with respect to $x$ twice and then multiplying the result by $u_{xx}$, we have
\bb\label{5.0.3}
\ba{l}
    \ds{\partial_t\left(\frac{u_{xx}^2+u_{xxx}^2}{2}\right) -\p_x\left(u_{xx}u_{txxx}+uu_{xx}u_{xxxx}+u_xu_{xx}u_{xxx}\right)}\\
    \\
    \ds{+u\p_x\left(\f{u_{xx}^2+u_{xxx}^2}{2}\right)+3u_xu_{xx}^2+u_xu_{xxx}^2=0.}
\ea
\ee
\end{itemize}

Note that
$$
\f{1}{2}\frac{u^2+u_x^2}{2}+\frac{u_x^2+u_{xx}^2}{2}+\frac{u_{xx}^2+u_{xxx}^2}{2}=\f{u^2+3u_x^2+4u_{xx}^2+2u_{xxx}^2}{4},
$$
which is a linear combination of the terms differentiated with respect to $t$ in \eqref{5.0.1}--\eqref{5.0.3}. Let us then define

\bb\label{5.0.4}
I[u]:=\f{1}{4}\int_\R (u^2+3u_x^2+4u_{xx}^2+2u_{xxx}^2)dx.
\ee

From \eqref{5.0.1}--\eqref{5.0.4}, we have
$$
\f{d}{dt}I[u]=-\int_\R \left(u_x^3+3u_xu_{xx}^2+u_xu_{xxx}^2\right)dx-\f{1}{2}\int_\R u\p_x\left(2u_x^2+2u_{xx}^2+u_{xxx}^2\right)dx. 
$$

Integrating by parts the last integral and assuming that $u_x>-\kappa$, we have
$$
\f{d}{dt}I[u]=\int_\R(-u_x)\left(2u_{xx}^2+\f{u_{xxx}^2}{2}\right)dx\leq 2\kappa\int_\R u_{xx}^2dx+\f{\kappa}{2}\int_\R u_{xxx}^2dx\leq 3\kappa I[u],
$$
wherefore, after using the Grönwall's inequality, we conclude that
$$
I[u]\leq I_0 e^{3\kappa t},
$$
for some constant $I_0>0$. The proof is concluded by noticing that $u\mapsto\sqrt{I[u]}$ is a norm equivalent to the $H^3(\R)-$norm, since $\|u\|_{H^3(\R)}/2\leq\sqrt{I[u]}\leq \|u\|_{H^3(\R)}$.
\end{proof}

{\bf Proof of Theorem \ref{teo2.4}.} If $m_0\in L^1(\R)$ does not change sign, then $\sign{(u)}=\sign{(m)}=\sign{(m_0)}=\sign{(u_0)}$ in view of Theorem \ref{teo2.1}. By Theorem \ref{teo3.1} $u_x$ is bounded from below. Theorem \ref{teo2.4} is then a consequence of Lemma \ref{lema5.1}.\hfill$\square$

{\bf Proof of Corollary \ref{cor2.3}.}
Integrating \eqref{5.0.1} with respect to $x$ we have
$$
\ba{lcl}
\ds{\f{1}{2}\f{d}{dt}\|u(t,\cdot)\|_{H^1(\R)}^2}&=&\ds{-\int_\R u(t,x)\p_x(u_x(t,x)^2)dx=\int_\R u_x(t,x)^3dx\leq\int_\R |u_x(t,x)|u_x(t,x)^2dx}\\
\\
&\leq&\ds{\|u_x(t,\cdot)\|_{\infty}\int_\R\left(u(t,x)^2+u_x(t,x)^2\right)dx=\|u_x(t,\cdot)\|_{\infty}\|u(t,\cdot)\|_{H^1(\R)}^2.}
\ea
$$

Therefore, after applying the Gronwall's inequality, we are forced to conclude that
\bb\label{5.0.5}
\|u(t,\cdot)\|_{H^1(\R)}\leq \|u_0\|_{H^1(\R)}e^{\int \|u_x(\tau,\cdot)\|_\infty d\tau}.
\ee

If $\lim\limits_{t\nnearrow T}\|u(t,\cdot)\|_{H^1(\R)}=+\infty,$
then necessarily we have 
$\lim\limits_{t\nnearrow T}\|u(t,\cdot)\|_{\infty}=+\infty$, whereas if $\lim\limits_{t\nnearrow T}\|u(t,\cdot)\|_{\infty}<+\infty$, 
then by \eqref{5.0.5} we are forced to conclude that $u_x(t,x)$ is bounded for each $t$ arbitrarily close to $T$.
\hfill$\square$

\section{Dynamics of solutions}\label{sec6}

In this section we investigate some solutions of equation \eqref{1.0.6} of the form
\bb\label{6.0.1}
u(t,x)=\sum_{j=1}^N u_j(t,x),
\ee
where each function $u_j(t,x)$ in \eqref{6.0.1} is at least continuous. Namely, we consider the following types of solutions:
\begin{itemize}
\item with pointed crest, in which their lateral derivatives are finite, but not equal. A typical solution is obtained by taking
\bb\label{6.0.2}
u_j(t,x)=p_j(t)e^{-|x-q_j(t)|},
\ee
where the functions $p_j,\,q_j,\,1\leq j\leq N$, are functions having first order derivatives, but they are not necessarily continuously differentiable.

This sort of solution is best known as multi-peakons or {\it N-peakons}, see \cite{anco,priigorjde,deg,holm-siam,holm-pla}. Of particular interest is the 1-peakon $u(t,x)=Ae^{-|x-ct|}$, for certain constants $A$ and $c$, which has a jump in the derivatives of $u$ along the curve $t\mapsto(t,ct)$;
\item with no jump in the derivatives, but anti-symmetric along certain curves. These solutions are called {\it cliffs}, see \cite{holm-siam,holm-pla}, and will be of the form 
\bb\label{6.0.3}
u_j(t,x)=c_j(t)+b_j(t)\,\sign{(x-p_j(t))}(1-e^{-|x-p_j(t)|}),
\ee
for certain functions $b_j(t)$, $c_j(t)$ and $p_j(t)$.
\end{itemize}

We note that all the ansatzes \eqref{6.0.2} and \eqref{6.0.3}, when substituted into \eqref{6.0.1}, will lead to a dynamical system to the corresponding unknown functions.

Let $y=y(t)\in\R^n$, $t\in\R$ and $F:\R^n\rightarrow\R^n$ be a function. We recall that a point $y_\ast\in\R^n$ is a critical point of the dynamical system $y'(t)=F(y(t))$ if $F(y_\ast)\equiv0$. Moreover, if $F$ is continuous and locally Lipschitz in a certain domain $\Omega$, then the Picard-Lindel\"of theorem (see \cite[page 10]{valls}) assures the existence of a unique solution for the problem $y'(t)=F(y(t)),\,\,\,y(t_0)=y_0$.

Finally, we observe that some of the functions $u=u(t,x)$ we want to consider here are continuous, and only continuous. Then, these solutions are to be understood as distributional ones. Henceforth, we shall use some facts about distributions. The Dirac delta distribution centered at a point $x_0$ is denoted by $\delta(x-x_0)$, while $\sign{(x)}$ means the sign distribution. It is related to the Dirac delta distribution by the relation $\sign'(\cdot)=2\delta(\cdot)$. For further details see \cite[Chapter 2]{vlad}, and \cite[Chapter 11]{bruno}. We also guide the reader to \cite{anco,himonas-procams} for further details since these references study similar solutions following the same approach as ours. In particular, in \cite{himonas-procams} several similar calculations as those in our paper are done with enough detail.

\subsection{$N-$peakons for $k\in\mathbb{N}$}

Let us assume that 
\bb\label{6.1.1}
u(t,x) = \sum\limits_{i=1}^Np_i(t)e^{-|x-q_i(t)|},
\ee
for certain functions $p_i=p_i(t)$ and $q_i=q_i(t)$, $i=1,\dots,N$, is a solution of \eqref{1.0.6}. We note that these solutions look like $N$ pulses, with amplitudes $p_i(t)$ and positions $q_i(t)$. They are called generically peakons, although sometimes peakon is refereed to pulses with positive amplitudes whereas those with negative amplitude are named antipeakons, see Figure \ref{fig1}. 
\begin{figure}[ht!]
\centering
\includegraphics[width=.5\linewidth]{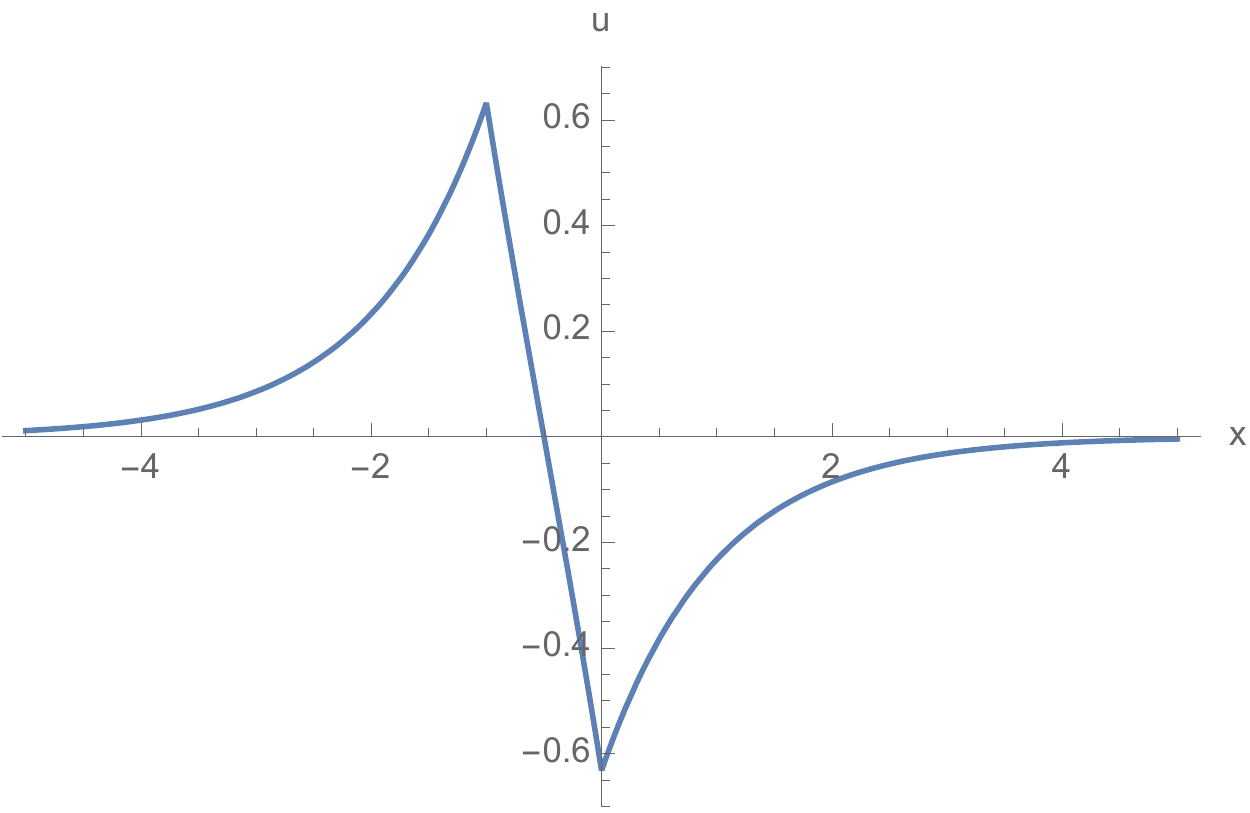}
  \caption{Function $u(t,x)=e^{-|x-t+1|}-e^{-|x+t|}$. We observe two pulses with opposite signs, one of then having positive amplitude (the one above the $x$-axis), corresponding to a peakon, and the other (below the $x$-axis), with negative amplitude, corresponding to an antipeakon.}
\label{fig1}
\end{figure}

Taking the distributional derivatives of $u$, we have
\bb\label{6.1.2}
m=2\sum\limits_{i=1}^Np_i\delta(x-q_i),
\ee
where the dependence on the variable $t$ was omitted for convenience, and
\bb\label{6.1.3}
m_x = 2\sum\limits_{i=1}^Np_i\delta'(x-q_i),\quad\quad m_t = 2\sum\limits_{i=1}^Np_i'\delta(x-q_i) - 2\sum\limits_{i=1}^Np_iq_i'\delta'(x-q_i).
\ee

Substituting \eqref{6.1.2} and \eqref{6.1.3} into \eqref{1.0.6} and after straightforward calculations we obtain
\bb\label{6.1.4}
\ba{lcl}
\ds{2\sum_{i=1}^N\left[p_i'-kp_i u^{k-1}(t,q_i)u_x(t,q_i)\right]\delta(x-q_i)}\\
\\\ds{-2\sum_{i=1}^N p_i\left [q_i'-u^k(t,q_i)\right]\delta'(x-q_i)=0.}
\ea
\ee

We can argue that the \eqref{6.1.4} must vanish identically, which would then imply that the coefficients of $\delta$ and $\delta'$ are 0. However, if we follow the steps in \cite[Subsection 6.2 ]{priigorjde} (the key idea is to use Lemma 1 of \cite[page 9]{fomin}), or the same steps as in \cite{anco,himonas-procams}, we can rigorously prove that
\bb\label{6.1.5}
\left\{
\ba{lcl}
p_i'&=&kp_i u(t,q_i)^{k-1}u_x(t,q_i),\\
\\
q_i'&=& u(t,q_i)^k,
\ea
\right.
\ee
where
$$
\left\{
\ba{lcl}
u(t,q_i)&=&\ds{\sum_{j=1}^Np_j(t)e^{-|q_i(t)-q_j(t)|}},\\
\\
u_x(t,q_i)&=&\ds{-\sum_{j=1}^N\sign{(q_i(t)-q_j(t))}p_j(t)e^{-|q_i(t)-q_j(t)|}}.
\ea
\right.
$$

A single peakon solution can be easily found. In fact, for the case of 1-peakon we have $u(t,x)=p(t)e^{-|x-q(t)|}$. Then \eqref{6.1.5} reads $p'=0$ and $q'=p^k$. Then, defining $p=c^{1/k}$, with $c>0$, and $q=ct+q_0$, we obtain
$$
u(t,x)=c^{1/k} e^{-|x-ct-q_0|},
$$
where $q_0$ is an arbitrary constant.

Other peakon solutions to \eqref{1.0.6} or, more specifically, multi-peakon solutions, can be found by solving \eqref{6.1.5}, see \cite{anco}, although a solution of the general case is far from a simple task. 

In what follows we also pay some attention to the particular case
\bb\label{6.1.8}
u(t,x)=p_1(t)e^{-|x-q_1(t)|}+p_2(t)e^{-|x-q_2(t)|}.
\ee

\subsubsection{2-peakon dynamics}
If we consider $N=2$ in \eqref{6.1.1}, then \eqref{6.1.8} gives (the dependence on $t$ will be omitted for convenience)
\bb\label{6.1.9}
\left\{\ba{lcl}
p_1'&=&-k\,\sign{(q_1-q_2)}\,p_1p_2(p_1+p_2e^{-|q_1-q_2|})^{k-1}e^{-|q_1-q_2|},\\
\\
p_2'&=&k\,\sign{(q_1-q_2)}\,p_1p_2(p_1e^{-|q_1-q_2|}+p_2)^{k-1}e^{-|q_1-q_2|},\\
\\
q_1'&=&(p_1+p_2e^{-|q_1-q_2|})^{k},\\
\\
q_2'&=&(p_1e^{-|q_1-q_2|}+p_2)^{k}.
\ea\right.
\ee
If we take $p_1=-p_2=p$, $q_1=-q_2=q$ and assume that $k$ is even, then the resulting set of equations implies that $p=0$ or $q=0$. In any case we would then obtain the trivial solution  $u(t,x)=0$. On the other hand, if we assume that $k$ is odd, the four-dimensional dynamical system \eqref{6.1.9} becomes
\bb\label{6.1.10}
\left\{\ba{lcl}
p'&=&k\,\sign{(2q)}\,p^{k+1}(1-e^{-2|q|})^{k-1}e^{-2|q|},\\
\\
q'&=&p^k(1-e^{-2|q|})^{k}.
\ea\right.
\ee
The critical points of the system \eqref{6.1.10} are $p=0$ or $q=0$, which again imply the trivial solution. Let $\mathcal{R}:=(0,\infty)\times(0,\infty)$, $p(0)=p_0$, $q(0)=q_0$ and $(p_0,q_0)\in{\mathcal{R}}$. For $(p,q)\in\mathcal{R}$, system \eqref{6.1.10} reduces to
\bb\label{6.1.11}
\left\{\ba{lcl}
p'&=&k\,p^{k+1}(1-e^{-2q})^{k-1}e^{-2q},\\
\\
q'&=&p^k(1-e^{-2q})^{k}.
\ea\right.
\ee
Since the function $(p,q)\mapsto (k\,p^{k+1}(1-e^{-2q})^{k-1}e^{-2q},p^k(1-e^{-2q})^{k})$ is $C^\infty$ on $\mathcal{R}$, we have granted the existence of an interval $I\subseteq\mathbb{R}$ such that $0\in I$, and a local solution $(p,q)$ to \eqref{6.1.11} subject to $\left.(p,q)\right|_{t=0}=(p_0,q_0)$ defined on $I$.
On $\mathcal{R}$ we note that $p'>0$ and $q'>0$, meaning that they are increasing smooth functions and, in particular, $p(t)>p_0$ and $q(t)>q_0$ for each $t\in I\cap(0,\infty)$.

We can easily use \eqref{6.1.11} to express $p$ as a function of $q$, that is,
\bb\label{6.1.12}
p(t)=p_0\left(\f{1-e^{-2q(t)}}{1-e^{-2q_0}}\right)^{k/2}.
\ee
If we substitute \eqref{6.1.12} into the second equation in \eqref{6.1.11} we can then integrate it and obtain $q$. However, the resulting expression is an implicit function and we do not write it here. Note, however, that the 2-peakon solution to this case is given by
$$
u(t,x)=p_0\left(\f{1-e^{-2q(t)}}{1-e^{-2q_0}}\right)^{k/2}\left(e^{-|x-q(t)|}-e^{-|x+q(t)|}\right).
$$
The solution above corresponds to a peakon/antipeakon solution.
\begin{remark}
We note that the choice $p_1=-p_2=p$ implies that the solution \eqref{6.1.8} satisfies the condition $\mathcal{H}_0[u]=0$, where $\mathcal{H}_0[u]$ is given by \eqref{1.0.8}. However, this is a conserved quantity for the equation only when $k=1$.
\end{remark}

\subsection{$N-$peakons for $k=-n,\,n\in\mathbb{N}$.} The general equations for the $N-$ peakon solutions are obtained directly by taking $k=-n$ in \eqref{6.1.5} and, therefore, we do not repeat the process again. 

In this case we can interpret equation \eqref{1.0.6} as 
\bb\label{6.2.1}
u^nm_t=m_x.
\ee

Regarding the $2-$peakon dynamics, likewise the previous subsection, if $n$ is even we would only have $u(t,x)\equiv0$ (note that this solution is admitted by \eqref{6.2.1}). However, for $n$ odd, proceeding similarly as before, we obtain
$$
p(t)=p_0\left(\f{1-e^{-2q_0}}{1-e^{-2q(t)}}\right)^{n/2},
$$
where $(1-2q)^n p^nq'=1$, $(p_0,q_0)\in\mathcal{R}$, and
$$
u(t,x)=p_0\left(\f{1-e^{-2q_0}}{1-e^{-2q(t)}}\right)^{n/2}\left(e^{-|x-q(t)|}-e^{-|x+q(t)|}\right).
$$

\subsubsection{2-peakon dynamics for the case $k=-1$}
We can explore the 2-peakons dynamics for $k=-1$ by using the conserved quantity \eqref{1.0.9}. 

Let us assume again that $u(t,x)=p_1(t)e^{-|x-q_1(t)|}+p_2(t)e^{-|x-q_2(t)|}$. If we impose that such solution has \eqref{1.0.9} as a conserved quantity, then we have
\bb\label{6.2.2}
0\leq\mathcal{ H}[u]=p_1^2+2p_1p_2e^{-|q_1-q_2|}+p_2^2.
\ee
Let $p_{10}=p_{1}(0)$, $p_{20}=p_{2}(0)$ and $q_0:=|q_1(0)-q_2(0)|$, which is nothing but the initial separation of the pulses. The conservation of \eqref{1.0.9} yields
\bb\label{6.2.3}
\mathcal{ H}[u(0,x)]=p_{10}^2+2p_{10}p_{20}e^{-q_0}+p_{20}^2=:\mathcal{ H}_0.
\ee
Then, equations \eqref{6.2.3} and \eqref{6.2.2} read
$$
p_1^2+2p_1p_2e^{-|q_1-q_2|}+p_2^2=\mathcal{ H}_0.
$$
If we assume that $(p_{10},p_{20})\neq(0,0)$, then $\mathcal{H}_0>0$. We observe that from \eqref{6.2.2} we have the estimates
\bb\label{6.2.4}
0\leq e^{-|q_1-q_2|}=\frac{\mathcal{ H}_0-p_1^2-p_2^2}{2p_1p_2}\leq 1.
\ee

Let us define
$$
\ba{lcl}
\ds{p_1+p_2e^{-|q_1-q_2|}=\f{\mathcal{H}_0+p_1^2-p_2^2}{2p_1}=:A_1},\\
\\
\ds{p_1e^{-|q_1-q_2|}+p_2=\f{\mathcal{H}_0-p_1^2+p_2^2}{2p_2}=:A_2}.
\ea
$$
Using $A_1$ and $A_2$ given above, system \eqref{6.1.9} with $k=-1$ reads
\bb\label{6.2.5}
\ba{lcl}
p_1'&=&\ds{\f{1}{2}\,\sign{(q_1-q_2)}\f{\mathcal{H}_0-p_1^2-p_2^2}{2 A_1^2}},\quad q_1'=\ds{\f{1}{A_1}},\\
\\
p_2'&=&\ds{-\f{1}{2}\,\sign{(q_1-q_2)}\f{\mathcal{H}_0-p_1^2-p_2^2}{2 A_2^2}},\quad 
q_2'=\ds{\f{1}{A_2}}.
\ea
\ee

System \eqref{6.2.5} does not have critical points. The last two equations cannot vanish, which implies that if we have solutions of the form \eqref{6.1.8} then they either have two pulses or degenerate into a 1-peakon solution. On the other hand, we may have $p_1'=p_2'=0$. This corresponds to one of the following situations:
\begin{itemize}
\item $q_1=q_2$. We have the superposition of the two peakons into a single one, meaning that the solution degenerates into a (one-)peakon
$$
u(t,x)=\f{1}{c}e^{-|x-ct-q|},
$$
or an antipeakon
$$
u(t,x)=-\f{1}{c}e^{-|x+ct-q|},
$$
where, in any case, $c>0$, and $q$ is a constant.
\item $q_1-q_2\rightarrow\infty$. This condition implies that $e^{-|q_1-q_2|}\approx0$ and thus $\mathcal{H}_0\approx p_1^2+p_2^2$, meaning that $p_1$ and $p_2$ are two constants and the pulses are infinitely separated. Defining $p_1=1/c_1$ and $p_2=1/c_2$, where $c_1$ and $c_2$ are two non-zero different constants. Then, for $t\gg 1$, we have
$$
\mathcal{H}=\f{1}{c^2_1}+\f{1}{c^2_2},\quad q_1(t)=c_1t+q_0,\quad q_2=c_2t,
$$
where $q_0$ is a constant of integration (and the corresponding constant to $q_2$ is conveniently taken as $0$). Note that we can consider $q_0$ as the separation of the pulses at $t=0$. Therefore, the asymptotic solution is given by
$$
u(t,x)=\f{1}{c_1}e^{-|x-c_1t-q_0|}+\f{1}{c_2}e^{-|x-c_2t|}.
$$

If both $c_1$ and $c_2$ are positive or negative, we have, respectively, 2-peakons or 2-anti-peakons, while if they have different signs we have a peakon and an anti-peakon, travelling in opposite directions. 


\end{itemize}



\subsection{Kink-type solutions for $k\in\mathbb{N}$}
Let us assume that 
\bb\label{6.3.1}
u_j(t,x)=c_j(t)+b_j(t)\sign{(x-p_j(t))}(1-e^{-|x-p_j(t)|})
\ee 
in \eqref{6.0.1}, for some functions $c_j$, $b_j$ and $p_j$. We again omit the dependence with respect to the independent variables for convenience.

If we denote $m_j=u_j-\p_x^2u_j$, we conclude that
$$\ba{lcl}
m_j&=&\ds{c_j+b_j\sign{(x-p_j)},\quad \p_tm_j=c_j'+b_j'\sign{(x-p_j)}-2b_jp_j'\delta(x-p_j)},\\
\\
\p_x m_j&=&2b_j\delta(x-\p_j).
\ea$$
Substituting the expressions above into \eqref{1.0.6}, we obtain
\bb\label{6.3.2}
\sum_{j=1}^N\left[c_j'+b_j'\sign{(x-p_j)}+2b_j\left(u(t,p_i)^k-p_i '\right)\delta(x-p_i)\right]=0.
\ee
Proceeding similarly as in the previous subsection, \eqref{6.3.2} results, for $j=1,\cdots,N$,
\bb\label{6.3.3}
\left\{\ba{lcl}
c_j'&=&0,\quad 
b_j'=0,\\
\\
p_j'&=&\ds{\left[\sum_{j=1}^N(c_j+b_j\sign(p_i-p_j)(1-e^{-|p_i-p_j|}))\right]^k.}
\ea\right.
\ee
In view of \eqref{6.3.1} we make the technical hypothesis that if for some $k\in\{1,\cdots,N\}$, $b_k(t)\equiv0$, then $p_k(t)\equiv0$.

System \eqref{6.3.3} directly implies that $c_j=const$ and $b_j=const$, $1\leq j\leq N$. If we assume that $N=2$, $p_1=p_2=p$, $0<c_1=c_2=c^{1/k}/2$, $b_1=b_2=b\neq0$, we conclude that $p(t)=ct+p_0$ and we have the solution
\bb\label{6.3.4}
u(t,x)=c^{1/k}+b\,\sign{(x-ct-p_0)}(1-e^{-|x-ct-p_0|}).
\ee
\begin{figure}[!htb]
    \centering
    \begin{minipage}{.49\textwidth}
        \centering
        \includegraphics[width=\linewidth, height=0.20\textheight]{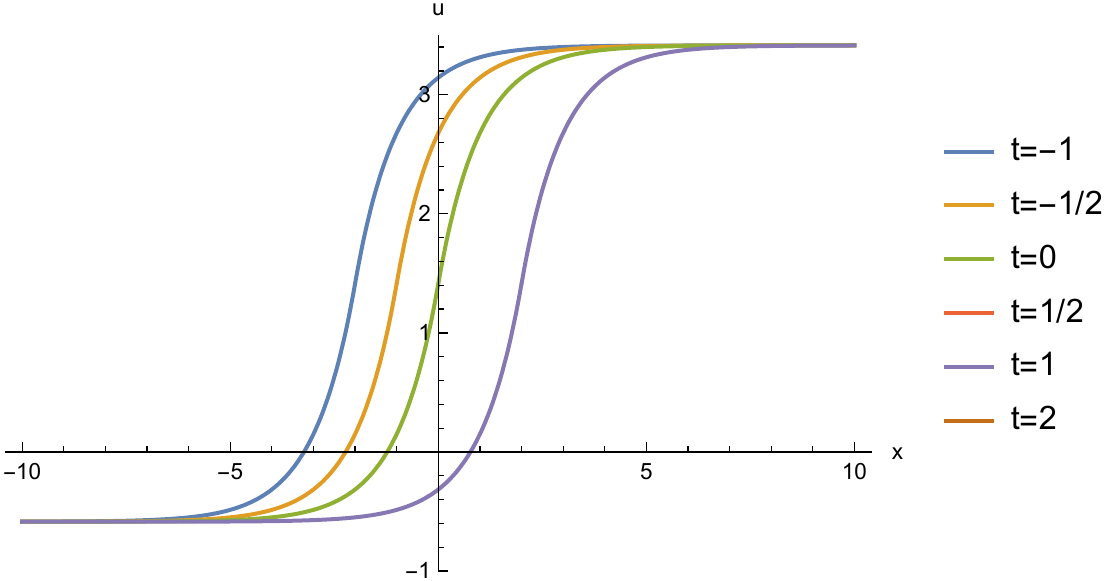}
        \caption{Behavior of the solution \eqref{6.3.4} with $c=2$, $b=2$ and $q_0=-1$.}
        \label{fig4}
    \end{minipage}\,\,
    \begin{minipage}{0.49\textwidth}
        \centering
        \includegraphics[width=\linewidth, height=0.2\textheight]{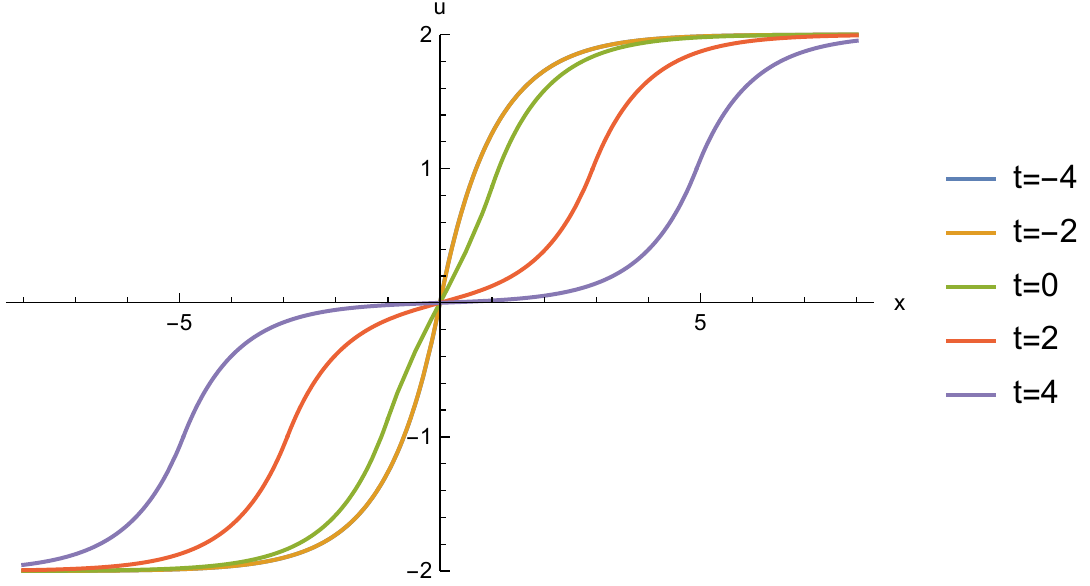}
        \caption{Behavior of the solution \eqref{6.3.7} with $p_0=1$.}
        \label{fig5}
    \end{minipage}
\end{figure}

Yet taking $N=2$, but $p_1=-p_2=p$, from \eqref{6.3.3} we obtain two equations:
$$
\ba{lcl}
p'&=&\left[c_1+c_2+b_1\,\sign(2p)(1-e^{-2|p|})\right]^k,\\
\\
p'&=&-\left[c_1+c_2-b_2\,\sign(2p)(1-e^{-2|p|})\right]^k.
\ea
$$
If we take $c_1+c_2=0$ and $b_2^k=(-1)^{k+1}b_1^k$, $p_0>0$, up to scaling in $t$, we have the following PVI in the region $\mathcal{R}:=(0,\infty)\times(0,\infty)$:
\bb\label{6.3.5}
\left\{
\ba{lcl}
p'&=&(1-e^{-2p})^k,\\
\\
p(0)&=&p_0.
\ea
\right.
\ee

We observe that $p'>0$ in \eqref{6.3.5}, which means that it a local increasing diffeomorphism. This implies that the solution of \eqref{6.3.5} will make \eqref{6.3.1} a monotonic and bounded function, which is nothing but a kink solution. Figures \ref{fig4} and \ref{fig5} show the typical behaviour of a kink solution.

It is worth mentioning that \eqref{6.3.5} has a unique local solution in $\mathcal{R}$ and, in particular, the solution of \eqref{6.3.5} can be implicitly given in terms of the hypergeometric function, since
\bb\label{6.3.6}
\int\f{dp}{(1-e^{-2p})^k}=(1-e^{-2p})^{-k}(1-e^{2x})^k {}_3F_2(k,k;k+1,e^{2p})+const.
\ee
For the case $k=1$ we can find the solution explicitly, namely,
$$
\int\f{dp}{1-e^{-2p}}=\f{1}{2}\ln(e^{2p}-1)+const.
$$
From this and \eqref{6.3.5} we conclude that
$$
p(t)=\f{1}{2}\ln\left[1+(e^{2p_0}-1)e^{2t}\right].
$$

For convenience, let us assume that $b_1=b_2=1$. Then our solution is
\bb\label{6.3.7}
\ba{lcl}
u(t,x)&=&\ds{\sign\left(x-\f{1}{2}\ln\left[1+(e^{2p_0}-1)e^{2t}\right]\right)(1-e^{-\left|x-\f{1}{2}\ln\left[1+(e^{2p_0}-1)e^{2t}\right]\right|})}\\
\\
&&+\ds{\sign\left(x+\f{1}{2}\ln\left[1+(e^{2p_0}-1)e^{2t}\right]\right)(1-e^{-\left|x+\f{1}{2}\ln\left[1+(e^{2p_0}-1)e^{2t}\right]\right|}).}
\ea
\ee

\section{Proof of theorems \ref{teo2.5} and \ref{teo2.6}}\label{sec7}

We recall that we can rewrite \eqref{1.0.6} as
\bb\label{7.0.1}
u_t+u^ku_x=F_t(x),
\ee
where
\bb\label{7.0.2}
F_t(x):=\f{k(k-1)}{2}\Lambda^{-2}(u^{k-2}u_x^3)(t,x)-\f{3k}{2}\p_x\Lambda^{-2}(u^{k-1}u_x^2)(t,x).
\ee

Another equation widely used throughout this section is the differential consequence of \eqref{7.0.1}
\bb\label{7.0.3}
u_{tx}+ku^{k-1}u_x^2+u^ku_{xx}=\p_x F_t(x).
\ee

We have some preliminary results regarding \eqref{7.0.2}. Let $T_0\in(0,T)$. Henceforth we denote the compact set $[0,T_0]\subseteq[0,T)$ by $\I$. 

\begin{proposition}\label{prop6.1}
If $u\in C^0(\I,H^s(\R))$, $s>3/2$, is a solution of \eqref{1.0.6}, then $u(t,\cdot)^nu_x(t,\cdot)^m\in L^1(\R)\cap L^\infty(\R)$, for all $m,\,n\in\mathbb{N}_0$. Moreover, $F_t(\cdot)\in L^1(\R)\cap L^\infty(\R)\cap C^1(\R)$.
\end{proposition}

\begin{proof}
Since $u(t,\cdot)\in H^s(\R)$ and $s>3/2$, then both $u(t,\cdot)$ and $u_x(t,\cdot)$ belong to $L^\infty(\R)$ in view of the Sobolev Embedding Theorem \cite[page 317, Proposition 1.3]{taylor}. Moreover, the algebra property  \cite[page 320, Exercise 6]{taylor} implies that $u(t,\cdot)^nu_x(t,\cdot)^m\in H^{s-1}(\R)$ and again due to the Sobolev Embedding Theorem they are bounded and continuous. Moreover, $$|u(t,x)^nu_x(t,x)^m|\leq\|u(t,)\|_\infty^{n-1}\|u_x(t,\cdot)\|_{\infty}^{m-1}|u(t,x)||u_x(t,x)|,$$ 
and the Hölder's inequality yields $u(t,\cdot)^nu_x(t,\cdot)^m\in L^1(\R)$.

It remains to analyze $F_t(\cdot)$. It is immediate from \eqref{7.0.2} that $F_t(\cdot)\in C^1(\R)$ since the operators $\Lambda^{-2}$ and $\p_x\Lambda^{-2}$ increases the regularity of the products, involving $u$ and $u_x$. The other part follows from the fact that $\Lambda^{-2}h=g\ast h$, $\p_x\Lambda^{-2}h=(\p_xg)\ast h$, and both $g$ and $\p_xg\in L^1(\R)$. 
\end{proof}

The demonstration of Proposition \ref{prop6.1} proves the following result.
\begin{corollary}\label{cor7.1}
Under the conditions in Proposition \ref{prop6.1}, $F_t(\cdot),\,u(t,\cdot)^nu_x(t,\cdot)^m\in L^p(\R)$, for all $1\leq p\leq \infty$, and all non-negative integers $m,\,n\in\mathbb{N}$.
\end{corollary}

\begin{proposition}\label{prop6.2}
Let $u\in C^0(\I,H^s(\R))$ be a solution of \eqref{1.0.6}, $F_t(\cdot)$ be the function \eqref{7.0.2} and $q\in\R$ such that
$$
\f{1}{2p}+\f{1}{q}=1.
$$
Then
\begin{enumerate}
    \item $\|u(t,\cdot)^{2p-1}\|_q=\|u(t,\cdot)\|_{2p}^{2p-1}$;
    \item
    $$\Big|\int_\R u(t,\cdot)^{2p-1}F_t(x)dx\Big|\leq \|u(t,\cdot)\|_{2p}^{2p-1}\|F_t(\cdot)\|_{2p.}$$
    \item
    $$\Big|\int_\R u_x(t,x)^{2p-1}\p_x F_t(x)dx\Big|\leq \|u_x(t,\cdot)\|_{2p}^{2p-1}\|\p_x F_t(\cdot)\|_{2p.}$$
\end{enumerate}
\end{proposition}

\begin{proof}
1.) Note that $q(2p-1)=2p$ and
$$
\|u(t,\cdot)^{2p-1}\|_{q}^q=\int|u(t,\cdot)|^{(2p-1)q}dx=\int u(t,\cdot)^{2p}dx=\|u(t,\cdot)\|_{2p}^{2p}.
$$
Therefore, $\|u(t,\cdot)^{2p-1}\|_q=\|u(t,\cdot)\|_{2p}^{2p/q}=\|u(t,\cdot)\|^{2p-1}_{2p}.$

Now we prove part 2. By Corollary \ref{cor7.1}, the Hölder's inequality and the part 1 of the Proposition, we have
$$
\ba{lcl}
\ds{\Big|\int_\R u(t,x)^{2p-1}F_t(x)dx\Big|}&\leq&\ds{\int_\R \Big|u(t,x)^{2p-1}F_t(x)\Big|dx\leq\|u(t,\cdot)^{2p-1}\|_{q}\|F_t(\cdot)\|_{2p}}\\
\\
&\leq&\ds{ \|u(t,\cdot)\|_{2p}^{2p-1}\|F_t(\cdot)\|_{2p.}}
\ea
$$

The demonstration of the last inequality is analogous and for this reason is omitted.
\end{proof}

Let $\theta$ be the number given by Theorem \ref{teo2.5}. For each integer $N$, let us consider the function
\bb\label{7.0.4}
\phi_N(x)=
\left\{
\ba{lcl}
e^{\theta |x|},&\text{if}&|x|<N,\\
\\
e^{\theta N},&\text{if}&|x|\geq N,
\ea
\right.
\ee

We note that $(\phi_N)_{N\in\mathbb{N}}\subseteq C^0(\R)$ is point-wisely convergent to the function $\phi(x)=e^{\theta|x|}$ and it is also $C^1(\R)$ almost everywhere (a.e.). Also, note that $\phi_N'(x)\leq\phi_N(x)$ a.e.

\begin{proposition}\label{prop6.3}
Let $u\in C^0(\I,H^s(\R))$, $s>3/2$, be a solution of \eqref{1.0.6}, $F_t(\cdot)$ be the function \eqref{7.0.2} and $\phi_N$ be the function \eqref{7.0.4}. Then
\begin{enumerate}
    \item There exists a constant $c_0>0$ such that 
    $$\phi_N(x)\int_\R\f{e^{-|x-y|}}{\phi_N(y)}dy\leq c_0;$$
    \item There exists a constant $c$, depending on $k$, $T_0$ and
    \bb\label{7.0.5}
    M^{1/k}:=\sup_{t\in\I}\|u(t,\cdot)\|_{H^s(\R)}+\sup_{t\in\I}\|u_x(t,\cdot)\|_{H^s(\R)}
    \ee
    such that
    $$
    |\phi_N(x)F_t(x)|\leq c(\|\phi_N(\cdot)u(t,\cdot)\|_{\infty}+\|\phi_N(\cdot)u_x(t,\cdot)\|_{\infty})
    $$
    and
    $$
    |\phi_N(x)\p_x F_t(x)|\leq c(\|\phi_N(\cdot)u(t,\cdot)\|_{\infty}+\|\phi_N(\cdot)u_x(t,\cdot)\|_{\infty})
    $$
\end{enumerate}
\end{proposition}

\begin{proof}
Part 1 follows from \cite[Eq. 2.19]{himjmp2014}.

For the first inequality of Part 2, note that
$$
\ba{lcl}
\Big|\phi_N(x)F_t(x)\Big|&\leq&\ds{\f{k(k-1)}{4}\phi_N(x)\int_\R e^{-|x-y|}\big(|u(t,y)|^{k-2}|u_x(t,y)|^3\big)dy}\\
\\
&&+\ds{\f{3k}{4}\int_\R e^{-|x-y|}\big(|u(t,y)|^{k-1}|u_x(t,y)|^2\big)dy}\\
\\
&=&\ds{\f{k(k-1)}{4}\phi_N(x)\int_\R \f{e^{-|x-y|}}{\phi_N(y)}\big(|\phi_N(y)u_x(t,y)||u(t,y)|^{k-2}|u_x(t,y)|^2\big)dy}\\
\\
&&+\ds{\f{3k}{4}\int_\R \f{e^{-|x-y|}}{\phi_N(y)}\big(|\phi_N(y)u_x(t,y)||u(t,y)|^{k-1}|u_x(t,y)|\big)dy}\\
\\
&\leq&\ds{\f{k(k-1)}{4}\|u(t,\cdot)\|_\infty^{k-2}\|u_x(t,\cdot)\|_\infty^{2}\|\phi_N(\cdot)u_x(t,\cdot)\|_\infty\Big(\phi_N(x)\int_\R \f{e^{-|x-y|}}{\phi_N(y)}dy\Big)}\\
\\
&&
+\ds{\f{3k}{4}\|u(t,\cdot)\|_\infty^{k-1}\|\phi_N(\cdot)u_x(t,\cdot)\|_\infty\|u_x(t,\cdot)\|_\infty\Big(\phi_N(x)\int_\R \f{e^{-|x-y|}}{\phi_N(y)}dy\Big)}\\
\ea
$$

Let $c:=(k(k-1)/4+3k/4)c_0M$. Then 
$$|\phi_N(x)F_t(x)|\leq c(\|\phi_N(\cdot)u(t,\cdot)\|_{\infty}+\|\phi_N(\cdot)u_x(t,\cdot)\|_{\infty}).$$

The other inequality is proved following the same steps and for this reason its demonstration is omitted.
\end{proof}

\subsection{Proof of Theorem \ref{teo2.5}}

Note that the theorem is proved if we can find a constant $L>0$ such that 
$$
|e^{\theta|x|}u(t,x)|+|e^{\theta|x|}u_x(t,x)|\leq L.
$$

We begin by multiplying \eqref{7.0.1} by $u^{2p-1}$ and integrating the result with respect to $x$ over $\R$, which yields
\bb\label{7.1.1}
\f{1}{2p}\f{d}{dt}\int_\R u(t,x)^{2p}dx=-\int_\R u(t,x)^{2p+k-1}u_x(t,x)dx+\int_\R u(t,x)^{2p-1}F_t(x)dx.
\ee

However, we have the following identity and estimate:
$$
\f{1}{2p}\f{d}{dt}\int_\R u(t,x)^{2p}dx=\|u(t,\cdot)\|_{2p}^{2p-1}\f{d}{dt}\|u(t,\cdot)\|_{2p},
$$
$$
\Big|\int_\R u(t,x)^{2p+k-1}u_x dx\Big|\leq \|u(t,\cdot)\|_\infty^{k-1}\|u_x(t,\cdot)\|_\infty\|u(t,\cdot)\|_{2p}^{2p}\leq M \|u(t,\cdot)\|_{2p}^{2p}.
$$

Therefore, from the relations above, Proposition \ref{prop6.2} and equation \eqref{7.1.1} we have
\bb\label{7.1.2}
\f{d}{dt}\|u(t,\cdot)\|_{2p}\leq M \|u(t,\cdot)\|_{2p}+\|F_t(\cdot)\|_{2p}.
\ee

Now, multiplying \eqref{7.0.3} by $u_x^{2p-1}$, integrating the result with respect to $x$ over $\R$, we obtain
\bb\label{7.1.3}
\ba{lcl}
\ds{\f{1}{2p}\f{d}{dt}\int_\R u_x(t,x)^{2p}dx}&+&\ds{\int_\R ku(t,x)^{k-1}u_x(t,x)^{2p+1}dx+\int_\R u(t,x)^ku_{x}(t,x)^{2p-1}u_{xx}(t,x)dx}\\
\\
&=&\ds{\int_\R u_x(t,x)^{2p-1}\p_x F_t(x)dx.}
\ea
\ee

We now observe that
$$
\f{1}{2p}\f{d}{dt}\int_\R u_x(t,x)^{2p}dx=\|u_x(t,\cdot)\|_{2p}^{2p-1}\f{d}{dt}\|u_x(t,\cdot)\|_{2p},
$$
$$
\Big|\int_\R ku(t,x)^{k-1}u_x(t,x)^{2p+1}dx\Big|\leq k M \|u_x(t,\cdot)\|_{2p}^{2p},
$$
$$
\int_\R u(t,x)^ku_{x}(t,x)^{2p-1}u_{xx}(t,x)dx=-\f{k}{2p}\int_\R u(t,x)^{k-1}u_x(t,x)^{2p+1}dx,
$$
which implies that
$$
\Big|\int_\R u(t,x)^ku_{x}(t,x)^{2p-1}u_{xx}(t,x)dx\Big|\leq kM\|u_x(t,\cdot)\|_{2p}^{2p}.
$$

Therefore, by Proposition \ref{prop6.2} and the relations above, from \eqref{7.1.3} we obtain
\bb\label{7.1.4}
\f{d}{dt}\|u_x(t,\cdot)\|_{2p}\leq 2k M \|u_x(t,\cdot)\|_{2p}+\|\p_xF_t(\cdot)\|_{2p}.
\ee

Now, multiplying \eqref{7.0.1} by $\phi_N(x)(\phi_N(x)u(t,x))^{2p-1}$, \eqref{7.0.3} by $\phi_N(x)(\phi_N(x)u_x(t,x))^{2p-1}$, and proceeding similarly as before for obtaining \eqref{7.1.2} and \eqref{7.1.4}, we obtain
\bb\label{7.1.5}
\f{d}{dt}\|\phi_N(\cdot)u(t,\cdot)\|_{2p}\leq M\|\phi_N(\cdot)u(t,\cdot)\|_{2p}+\|\phi_N(\cdot)F_t(\cdot)\|_{2p}
\ee
and 
\bb\label{7.1.6}
\f{d}{dt}\|\phi_N(\cdot)u_x(t,\cdot)\|_{2p}\leq (k+1) M\|\phi_N(\cdot)u_x(t,\cdot)\|_{2p}+\|\phi_N(\cdot)\p_xF_t(\cdot)\|_{2p}
\ee

Let $U(t):=\|\phi_N(\cdot)u(t,\cdot)\|_{2p}+\|\phi_N(\cdot)u_x(t,\cdot)\|_{2p}$. Adding \eqref{7.1.5} and \eqref{7.1.6}, we obtain
$$
\f{d}{dt}U(t)\leq (k+1)M U(t)+\|\phi_N(\cdot)F_t(\cdot)\|_{2p}+\|\phi_N(\cdot)\p_xF_t(\cdot)\|_{2p}
$$

Noticing that $e^{-(k+1)Mt}\leq 1$ and $e^{(k+1)Mt}\leq e^{(k+1)MT_0}$, application of Gronwall's inequality gives
$$U(t)\leq e^{(k+1)MT_0}U(0)+e^{(k+1)MT_0}\int_0^t\left(\|\phi_N(\cdot)F_\tau(\cdot)\|_{2p}+\|\phi_N(\cdot)\p_xF_\tau(\cdot)\|_{2p}\right)d\tau.$$

Recalling that $U(t):=\|\phi_N(\cdot)u(t,\cdot)\|_{2p}+\|\phi_N(\cdot)u_x(t,\cdot)\|_{2p}$, taking the limit $p\rightarrow\infty$ in the last inequality, we have
\bb\label{7.1.7}
\ba{lcl}
\|\phi_N(\cdot)u(t,\cdot)\|_\infty+\|\phi_N(\cdot)u_x(t,\cdot)\|_\infty&\leq& e^{(k+1)MT_0}\Big[\|\phi_N(\cdot)u_0(\cdot)\|_\infty+\|\phi_N(\cdot)u_0'(\cdot)\|_\infty\\
\\
&&\ds{+\int_0^t\left(\|\phi_N(\cdot)F_\tau(\cdot)\|_{\infty}+\|\phi_N(\cdot)\p_xF_\tau(\cdot)\|_{\infty}\right)d\tau}
\Big].
\ea
\ee

We note that from Proposition \ref{prop6.3} we have the estimate
$$
\|\phi_N(\cdot)F_\tau(\cdot)\|_{\infty}+\|\phi_N(\cdot)\p_xF_\tau(\cdot)\|_{\infty}\leq c(\|\phi_N(\cdot)u(t,\cdot)\|_{\infty}+\|\phi_N(\cdot)u_x(t,\cdot)\|_{\infty}),
$$
for some constant $c>0$. Therefore, the estimate above jointly with \eqref{7.1.7} imply
\bb\label{7.1.8}
\|\phi_N(\cdot)u(t,\cdot)\|_\infty+\|\phi_N(\cdot)u_x(t,\cdot)\|_\infty\leq c(\|\phi_N(\cdot)u_0(\cdot)\|_{\infty}+\|\phi_N(\cdot)u_0'(\cdot)\|_{\infty}).
\ee

We recall that both $|u_0|$ and $|u_0'|$ are $ O(e^{-\theta|x|})$, which mean that $e^{\theta|x|}u_0(x)$ and $e^{\theta|x|}u_0'(x)$ are bounded for $|x|\gg1$. Therefore, $\max\{u_0(x),e^{\theta|x|}u_0(x)\}$ is bounded for any $x\in\R$. As a consequence, we have
\bb\label{7.1.9}
\ba{lcl}
\ds{|\phi_N(x)u_0(x)|}&\leq&\ds{|\max\{1,e^{\theta |x|}\}u_0(x)|\leq \|u_0(\cdot)\max\{1,e^{\theta|\cdot|}\}\|_\infty=:L_1.}
\ea
\ee

Analogously, we conclude that 
\bb\label{7.1.10}
\ba{lcl}
\ds{|\phi_N(x)u'_0(x)|}&\leq&\ds{\max\{1,e^{\theta |x|}\}|u'_0(x)|\leq\|u'_0(\cdot)\max\{1,e^{\theta|\cdot|}\}\|_\infty=:L_2.}
\ea
\ee

Therefore, \eqref{7.1.9} and \eqref{7.1.10} imply that the sum $\|\phi_N(\cdot)u(t,\cdot)\|_\infty+\|\phi_N(\cdot)u_x(t,\cdot)\|_\infty$ is bounded from above by $L:=c(L_1+L_2)$. Since $L$ does not depend on $N$ and $t$, taking the limit $N\rightarrow\infty$ in \eqref{7.1.8} we conclude that
$$
\ba{lcl}
\ds{|e^{\theta|x|}u(t,x)|+|e^{\theta|x|}u_x(t,x)|}&=&\ds{\lim_{N\rightarrow\infty}\left(|\phi_N(x)u(t,x)|+|\phi_N(x)u_x(t,x)|\right)}\\
\\
&\leq&\ds{\lim_{N\rightarrow\infty}\left(\|\phi_N(\cdot)u(t,\cdot)\|_\infty+\|\phi_N(\cdot)u_x(t,\cdot)\|_\infty\right)\leq L}.
\ea
$$
\subsection{Proof of Theorem \ref{teo2.6}}

We begin this subsection with a technical and useful result:

\begin{proposition}\label{prop7.4}
Assume that $u\in C^0(\I,H^s(\R))$, $s\geq3$, is a non-zero solution of \eqref{1.0.10}, $\al$ and $t_1$ as in Theorem \ref{teo2.6}. Suppose that at least one of the following conditions is satisfied:
\begin{enumerate}
    \item $k=1$;
    \item $k$ is even and $m_0(x)\geq0$, for all $x\in\R$;
    \item $k$ is odd and either $m_0(x)\geq0$ or $m_0(x)\leq0$, for all $x\in\R$.
\end{enumerate}  
Then there exists a constant $\kappa>0$ such that the function defined in \eqref{7.0.2} satisfies the inequality.
$$
\lim_{x\rightarrow\infty}\f{1}{e^{-x}}\int_0^{t_1} F_\tau(x)d\tau\geq \kappa .
$$

In particular, $F_t(x)\sim O(e^{-x})$, but not $o(e^{-x})$.
\end{proposition}

\begin{proof} Let us first fix $t\in\I$. From \eqref{7.0.2} we have
$$
\ba{lcl}
F_t(x)&=&\ds{e^{-x}\Big[\f{k(k-1)}{4}\int_{-\infty}^x e^y(u^{k-2}u_x^3)(t,y)dy+\f{3k}{4}\int_{-\infty}^xe^y(u^{k-1}u_x^2)(t,y)dy\Big]}\\
\\
 &&+\ds{e^{x}\Big[\f{k(k-1)}{4}\int^{\infty}_x e^{-y}(u^{k-2}u_x^3)(t,y)dy-\f{3k}{4}\int^{\infty}_xe^{-y}(u^{k-1}u_x^2)(t,y)dy\Big]}
\ea
$$

Since $|u_0(x)|\sim o(e^{-x})$, we then have $|u_0(x)|\sim O(e^{-\al x})$. By Theorem \ref{teo2.5} both $u$ and $u_x$ are $O(e^{-\al x})$ and, therefore,
$$
u^{k-2}u_x^3,\,\,u^{k-1}u_x^2\sim O(e^{-\al(k+1)x})\Rightarrow u^{k-2}u_x^3,\,\,u^{k-1}u_x^2\sim o(e^{-x}).
$$

As a consequence
$$
\tilde{G}_t(x):=e^{x}\Big(\f{k(k-1)}{4}\int^{\infty}_x e^y(u^{k-2}u_x^3)(t,y)dy-\f{3k}{4}\int^{\infty}_xe^{-y}(u^{k-1}u_x^2)(t,y)dy\Big) \sim e^x\,o(e^{-2x})\sim o(e^{-x}).
$$

Let
\bb\label{7.2.1}
G_t(x):=e^{-x}\Big[\f{k(k-1)}{4}\int_{-\infty}^x e^y(u^{k-2}u_x^3)(t,y)dy+\f{3k}{4}\int_{-\infty}^xe^y(u^{k-1}u_x^2)(t,y)dy\Big]
\ee
and
$$
H(x):=\int_0^{t_1}G_\tau(x)dx.
$$

We observe that $$
F_t(x)=\tilde{G}_t(x)+G_t(x)
$$
and we prove the result if we show that $H(x)\geq\kappa e^{-x}$, for some $\kappa>0$. We divide our proof in three different cases: $k=1$ and $k\geq 2$, whereas the last case is subdivided into $k$ even or $k$ odd. 
\begin{itemize}
    \item[(a)]{\bf $k=1$.} We firstly observe that
    $$G_t(x)=\f{3}{4}e^{-x}\int_{-\infty}^x e^yu_x(t,y)^2 dy$$
    and
    $$
    \int_0^{t_1} G_\tau (x)d\tau=\f{3}{4}e^{-x}\int_{-\infty}^x e^y\mu(y)dy,
    $$
    where 
    $$\mu(y):=\int_0^{t_1}u_x(\tau,y)^2d\tau.$$
    Since $u$ is not identically zero we conclude that $\mu(\cdot)$ is a continuous, non-negative function and non-identically zero, so that for $x$ large enough, we have a constant $k_1>0$ such that
    $$\int_{-\infty}^xe^y\mu(y) \geq k_1.$$
    
    Therefore, taking $\kappa:=3 k_1/4$, we have $H(x)\geq \kappa e^{-x}$.
    \item[(b)] {\bf $k$ is even and $m_0(x)\geq 0$, for all $x\in\R$.} In this case, let $c_k:=\min\{k(k-1)/4,3k/4)\}$. In particular, note that $c_k>0$. From \eqref{7.2.1} we have
    $$
    G_t(x)\geq c_k e^{-x}\int_{-\infty}^x e^y u_x^2(t,y)u(t,y)^{k-2}(u+u_x)(t,y)dy
    $$
    and
    $$
    H(x)\geq c_k e^{-x}\int_{-\infty}^x e^y \mu(y)dy,
    $$
    where
    $$
    \mu(y):=\int_0^{t_1}u_x^2(\tau,y)u(\tau,y)^{k-2}(u+u_x)(\tau,y)d\tau
    $$
    Since $k$ is even and the initial data is non-negative, by Theorem \ref{teo2.1} and \eqref{4.1.2} we conclude that $e^y u_x^2(t,y)u(t,y)^{k-2}(u+u_x)(t,y)\geq0$ and an argument analogous as the previous case shows that  
    $$\int_{-\infty}^x e^y \mu(y)dy\geq k_2,$$
    for some positive constant $k_2$. Then, defining $\kappa:=c_k k_2$, we obtain $H(x)\geq \kappa e^{-x}$.
    
    \item[(c)]{\bf $k$ is odd and $m_0(x)\geq0$ or $m_0(x)\leq0$, for all $x\in\R$.} Since $k$ is odd and the initial data is either non-negative or non-positive, then $u(t,y)^{k-2}(u+u_x)(t,y)\geq0$. By Theorem \ref{teo2.1} and equation \eqref{4.1.2} we can once find a constant $k_3>0$ such that 
    $$\int_{-\infty}^x e^y \mu(y)dy\geq k_3,$$
    and then, defining $\kappa:=c_k k_2$, we conclude that $H(x)\geq\kappa e^{-x}$.
\end{itemize}

In any circumstance we have that $H(x)\geq \kappa e^{-x}$, for some constant $\kappa>0$. Then 
$$
\int_0^{t_1}F_\tau(x)dx=o(e^{-x})+H(x).
$$

Since $e^x H(x)\geq\kappa$ for $x\gg1$, we obtain the result.
\end{proof}

{\bf Proof of Theorem \ref{teo2.6}}. 

Similarly as in Proposition \ref{prop7.4}, the fact that $|u_0(x)|\sim o(e^{-x})$ implies $|u_0(x)|\sim O(e^{-\al x})$ and from Theorem \ref{teo2.5} we have
$$
|u^{k}u_x|\sim O(e^{-\al(k+1)x})\Rightarrow |u^{k}u_x|\sim o(e^{-x}).
$$

Also, since both $|u_0(x)|$ and $|u(t_1,x)|$ are $o(e^{-x})$. Therefore, integrating \eqref{7.0.1} with respect to $t$ from $0$ to $t_1$, we obtain
\bb\label{7.2.2}
\int_0^{t_1}F_\tau (x)d\tau=\underbrace{u(t_1,x)-u_0(x)}_{\sim o(e^{-x})}+\underbrace{\int_0^{t_1}(u^ku_x)(\tau,x)d\tau}_{\sim o(e^{-x})}\sim o(e^{-x})
\ee

If we could find $(t_0,x_0)$ such that $u(t_0,x_0)\neq0$, then \eqref{7.2.2} would contradict Proposition \ref{prop7.4}.

\section{Discussion}\label{sec8}

Theorem \ref{teo2.1}, among other results, shows that if the initial momentum $m_0\in H^1(\R)$ is compactly supported, then this property persists for any value of $t$ as long as the corresponding solution $u$ of the equation
\bb\label{8.0.1}
u_t-u_{txx}+u^ku_x-u^ku_{xxx}=0
\ee
exists. We observe that the local well-posedness of \eqref{8.0.1} subject to $u(0,x)=u_0(x)$ is granted by \cite[Corollary 2.1]{yan}, see also \cite[Theorem 1.1]{himjmp2014}.

For $k=1$, \cite[Theorem 3.1]{zhou} shows that if $u_0\in H^3(\R)$ is compactly supported, then the same does not hold for the corresponding solution $u(t,x)$, $t>0$. In \cite{himjmp2014} the authors studied \eqref{1.0.5} (with $c=1$) and they showed that its solutions, subject to $u(0,x)=u_0(x)$ cannot be compactly supported for any $t>0$ as long as $u_0$ is not trivial, that is, non-zero. However, the conditions imposed on the equation in \cite{himjmp2014} cannot cover \eqref{8.0.1}.

Our Theorem \ref{teo2.6}, whose demonstration is strongly dependent on the results proved in Theorem \ref{teo2.5}, shed light to this point and actually, generalizes the results proved in \cite{himjmp2014,zhou} in this matter, to \eqref{8.0.1}.

The results of theorems \ref{teo2.2}, \ref{teo2.3} and \ref{teo2.6}, in fact, answer some open questions raised in \cite{himjmp2014} (and reproduces in the Introduction) about unique continuation results for \eqref{1.0.5} with $c=1$ (we answered the question for the case $b=0$). Theorems \ref{teo2.2} and \ref{teo2.3} are restrict to the case $k=1$, whereas Theorem \ref{teo2.6} is proved for general values of $k$, but with conditions on the initial data.

Observe that Theorem \ref{teo2.3} improves the results of Theorem \ref{teo2.2}. Actually, while in Theorem \ref{teo2.2} we requested that the solution $u$ vanishes on an open set of the type $(0,T)\times I$, in Theorem \ref{teo2.3} we requested that $u$ would vanish on $(t_0,t_1)\times I\subseteq (0,T)\times\R$, for some open interval $I$. The price paid to relax the condition in Theorem \ref{teo2.2} is the imposition that the initial momentum does not change its sign. As a consequence of this hypothesis, Theorem \ref{teo3.1} assures the conservation of both $\|u(t,\cdot)\|_1$ and $\|m(t,\cdot)\|_1$. The proof of Theorem \ref{teo2.3} has two main pillars that consists on the use of the ideas introduced in \cite{linares} combined with the use of a conserved quantity, as pointed in \cite{igor-dgh}, see also \cite{pri-mon} for further discussions and geometrical meaning of this approach. It is worth mentioning that recently one of us has studied \eqref{1.0.10} in Gevrey spaces, see \cite{pri}.

The answer given by theorems \ref{teo2.2} and \ref{teo2.3} are novel and innovative, since they are only possible due to recent developments about unique continuation and persistence properties for the solutions of some shallow water model recently developed in \cite{igor-dgh,linares}, see also \cite{pri-mon,igor-arxiv1,igor-arxiv2}. These techniques are essentially geometric \cite{pri-mon} and based on physical aspects of the models \cite{igor-dgh,igor-arxiv1}.

We also generalized some results from \cite{himjmp2014,yan,zhou} to \eqref{8.0.1} with $k=1$ and subject to $u(0,x)=u_0(x)$, see Theorem \ref{teo2.4}. The key ingredient for proving this result is the conditions $m_0\in L^1(\R)\cap H^1(\R)$ and it does not change sign. While the latter implies that $u_0$ also does not change sign, the former has a consequence the fact that the corresponding solution $u$ has the $x-$derivative bounded from below by $-\|m_0\|_{1}$, which essentially reduces the demonstration of Theorem \ref{teo2.4} to the proof of Lemma \ref{lema5.1}. For its turn, such lemma is proved using the relations \eqref{5.0.1}--\eqref{5.0.3}.

The identity \eqref{5.0.1} has a very important consequence: it gives a necessary condition for the wave-breaking of the solutions of \eqref{8.0.1} with $k=1$, see \eqref{5.0.5}, which proves Corollary \ref{cor2.3}. We, however, are unable to find sufficient conditions for this blow-up. This make us point out the following provocation:

\begin{conjecture}
Let $k=1$, $u_0\in H^s(\R)$, $s>3/2$, and $u$ the corresponding solution to \eqref{8.0.1} with lifespan $T$. Then $u$ breaks at finite time if and only if $\lim\limits_{t\nnearrow T}|u_x(t,x)|=+\infty$.
\end{conjecture}

Also, we would like to point out that the question whether \eqref{8.0.1}, with $k=2, 3,4,\cdots$, admits global solutions $u$ subject to $u(0,x)=u_0(x)$, for a suitable choice of $s$, remains an open problem. 

We also studied some solutions of the equation \eqref{1.0.6}, namely, multi-peakon and kink-type solutions. We describe the dynamics of 2-peakon solutions for odd values of $k$. A very interesting result reported here is the case $k=-1$, when the have the conservation of the $H^1(\R)-$norm of the solutions of \eqref{1.0.6} with $k=-1$ is used to give a better description of the 2-peakon dynamics. We similarly make a detailed description of the peakon/antipeakon dynamics when $k=1$ compatible with the conserved quantity \eqref{1.0.8}.

Regarding kink-type solutions, we presented a picture of their dynamics, found some explicit solutions and also described the 2-kink solutions of the system \eqref{6.3.5}. Although the general solution is given in terms of the hypergeometric function, see \eqref{6.3.6}, for the case $k=1$ we find the 1-parameter explicit solution \eqref{6.3.7}, where the parameter is nothing but the initial condition of the Cauchy problem \eqref{6.3.5}. Indeed, we recover the results due to Xia and Qiao \cite{qiao} for the equation $m_t+u m_x=0$ to \eqref{1.0.6} with $k\in\mathbb{Z}$.

\section{Conclusion}
We studied the Cauchy problem \eqref{1.0.10} and also persistence properties of the solutions of the equation in \eqref{1.0.10}. The main results of the paper are given in Section \ref{sec2}, where we reported our main contributions, but not all, regarding \eqref{1.0.10}. We observe that some of our results answer questions pointed out by Himonas and Thompson \cite{himjmp2014}, as well as we generalized some results in \cite{himjmp2014,yan,zhou} regarding the equation \eqref{1.0.5} (eventually with some particular choices of the paramters) to the equation \eqref{8.0.1}. We also generalized the study of peakon and kink solutions made in \cite{qiao} for \eqref{8.0.1} with $k=1$ for \eqref{8.0.1} for $k\in\mathbb{Z}\setminus\{0\}$.

\section*{Acknowledgements}

The work of I. L. Freire is supported by CNPq (grants 308516/2016-8 and 404912/2016-8). P. L. da Silva  would like to thank FAPESP (grant number 2019/23688-4) for the financial support.

\end{document}